\documentclass[12pt]{article}
\usepackage[utf8]{inputenc}
\usepackage{float}
\usepackage{latexsym,graphicx}
\usepackage[dvipsnames]{xcolor}
\usepackage{amsmath}
\usepackage{amssymb}
\usepackage{amscd}
\usepackage{amsthm}
\usepackage{amsopn}
\usepackage[shortlabels]{enumitem}

\newtheorem{theorem}{Theorem}[section]
\newtheorem{lemma}[theorem]{Lemma}
\newtheorem{proposition}[theorem]{Proposition}

\theoremstyle{definition}

\newcommand{\footremember}[2]{%
    \footnote{#2}
    \newcounter{#1}
    \setcounter{#1}{\value{footnote}}%
}
\newcommand{\footrecall}[1]{%
    \footnotemark[\value{#1}]%
}

\newcommand\xqed[1]{%
  \leavevmode\unskip\penalty9999 \hbox{}\nobreak\hfill
  \quad\hbox{#1}}
\theoremstyle{definition}
\newtheorem{xdefinition}[theorem]{Definition}
\newenvironment{definition}{\begin{xdefinition}}{\xqed{$\triangle$}\end{xdefinition}}
\theoremstyle{remark}
\newtheorem{xremark}[theorem]{Remark}
\newenvironment{remark}{\begin{xremark}}{\xqed{$\triangle$}\end{xremark}}
\newtheorem{xexample}[theorem]{Example}
\newenvironment{example}{\begin{xexample}}{\xqed{$\triangle$}\end{xexample}}

\makeatletter

\@addtoreset{equation}{section}
\makeatother

\newcommand{\Spec}{\operatorname{Spec}}
\newcommand{\jvRe}{\operatorname{Re}}

\newcommand{\ii}{i}
\newcommand{\dd}{d}
\newcommand{\ee}{e} 
\newcommand{\BigO}{\mathcal{O}}

\theoremstyle{definition}

\title{Discussing  semigroup bounds with resolvent estimates}
\author{%
  B.~Helffer\footremember{Nantes}{Nantes Université, Laboratoire de
Mathématiques Jean Leray, LMJL, F-44000 Nantes, France. Bernard.Helffer@univ-nantes.fr / Joseph.Viola@univ-nantes.fr}%
  \and J.~Sj\"ostrand\footremember{Bourgogne}{Institut de Mathématiques de Bourgogne, UMR 5584 CNRS, Universite Bourgogne Franche-Comte,
F21000 Dijon Cedex France. Johannes.Sjostrand@u-bourgogne.fr}%
  \and J.~Viola\footrecall{Nantes}
  }

\date{\today}

\begin{document}

\bibliographystyle{plain}

\maketitle

\begin{abstract}
The purpose of this paper is to revisit the proof of the
Gearhart-Pr\"uss-Huang-Greiner theorem for a semigroup $S(t)$, following the
general idea of the proofs that we have seen in the literature and to
get an explicit estimate on the operator norm of $S(t)$ in terms of bounds on the resolvent of the generator. In \cite{HelSj} by the first two authors, this was done and some applications in semiclassical analysis were given. Some of these results have been subsequently published in two books written by the two first authors \cite{He1,Sj2}. A second work \cite{HeSj21} by the first two authors presents new improvements partially motivated by a paper of D.~Wei \cite{W}.

In this third paper, we continue the discussion on whether the aforementioned results are optimal, and whether one can improve these results through iteration. Numerical computations will illustrate some of the abstract results.
\end{abstract}
\vskip2cm\noindent
{\small {\bf 2020 Mathematics Subject Classification.--} 47D03, 44A10, 49K99.} 
\par\smallskip\noindent
{\small {\bf Key words and phrases.--} Semigroup, resolvent, optimal
  bounds, Riccati equation.} 
\newpage 

 \tableofcontents 
\newpage 

\section{Review of some recent results}

\label{int}
\subsection{Introduction}
We start by recalling quantitative versions  of the
Gearhart-Pr\"uss-Huang-Greiner theorem obtained since 2010 (see  \cite{HelSj,He1,Sj,Sj2,HeSj21}).  We also mention more recent contributions which use or are connected with these results \cite{Ar,BT,Ei,EZ,LY,RV,R,Wa,Bou}.

Throughout, we let 
\[
	[0,+\infty [\ni
t\mapsto S(t)\in {\mathcal L}({\mathcal H},{\mathcal H})
\]
denote a strongly continuous semigroup of operators with $S(0) = I$ acting on some complex Hilbert space $\mathcal{H}$. The norm $\|S(t)\|$ will refer to the norm of $S(t)$ as an operator on $\mathcal{H}$, and $A$ will refer to the generator of $S(t)$, so that formally $S(t)=\exp tA$. Recall (cf.~\cite[Chapter II]{EnNa07} or \cite{Paz}) that $A$ is closed and densely defined. We let $D(A)$ denote the domain of definition of $A$. 

By the Banach-Steinhaus theorem, $\sup_J\Vert S(t)\Vert$ is bounded for every compact interval $J\subset [0,+\infty [$. Using the semigroup property it follows easily that there exist $M \ge 1$ and $\omega_0 \in \mathbb{R}$ such that $S(t)$ has the property 
\begin{equation}\label{int.1}
P(M,\omega_0 ):\quad \Vert S(t)\Vert\le Me^{\omega_0 t},\ t\ge 0.
\end{equation}

We also recall (\cite[Theorem II.1.10]{EnNa07}) that 
\begin{equation}\label{int.2}
(z-A)^{-1}=\int_0^\infty S(t)e^{-tz}dt,\quad \Vert (z-A)^{-1}\Vert \le
\frac{M}{\Re z-\omega_0}\,,
\end{equation}
when $P(M,\omega_0 )$ holds and $z$ belongs to the open half-plane $\Re z> \omega_0 $.

\par We now recall the Gearhart-Pr\"uss-Huang-Greiner theorem, see
\cite[Theorem V.I.11]{EnNa07} or \cite[Theorem 19.1]{TrEm}:
\begin{theorem}\label{int1}~
\begin{enumerate}[(a)]
\item\label{it:GP1} Assume that $\Vert (z-A)^{-1}\Vert$ is uniformly bounded in the
half-plane $\Re z\ge \omega $. Then there exists a constant $M>0$ such
that $P(M,\omega )$ holds.
\item If $P(M,\omega )$ holds, then for every $\alpha >\omega $, $\Vert
(z-A)^{-1}\Vert$ is uniformly bounded in the half-plane $\Re z\ge
\alpha $.
\end{enumerate}
\end{theorem}

The purpose of \cite{HelSj} and \cite{HeSj21} was to revisit the proof of \ref{it:GP1} by getting an explicit $t$-dependent estimate on $e^{-\omega t}\Vert S(t) \Vert$, implying explicit bounds on $M$. 

To state the relevant results, we introduce a quantity $r(\omega)$ bounding $\|(z-A)^{-1}\|$ in the half-plane $\{\Re z \ge \omega\}$.
\begin{definition}
\begin{equation}\label{def:r}
	r(\omega) = \left(\sup_{\jvRe z > \omega}\|(z - A)^{-1}\|\right)^{-1},
\end{equation}
with the usual conventions that if $z \in \Spec A$ then $\|(z-A)^{-1}\| = +\infty$ and that formally $\frac{1}{+\infty} = 0$. 
\end{definition}

Clearly $r(\omega) \geq 0$ and $r(\omega)$ is increasing. We define
\begin{equation}\label{eq:omega_1}
	\omega_1 = \inf\{\omega \in \Bbb{R} \::\: r(\omega) > 0\}.
\end{equation}

With the triangle inequality we can easily show that  for every $\omega \in ]\omega_1,\infty [$, we have $\omega
-r(\omega )\ge \omega _1$ and for $\omega '\in [\omega -r(\omega
),\omega ]$ we have 
\begin{equation}\label{eq:monot}
r(\omega ')\ge r(\omega )-(\omega -\omega ').
\end{equation}

Note that in \cite{HelSj} (Remark 1.4), sufficient conditions are given to obtain that the $\sup$ appearing in the definition of $\omega$ is attained on $\jvRe z =\omega$.

The results discussed in this work are of the following form.

\begin{quote}
{\it Given a semigroup $S(t)$ of generator $A$, $\omega$ such that $r(\omega) >0$  and a  function $m:[0, +\infty[ \to [0, +\infty[$  such that
\[
	\|S(t)\| \leq m(t), \quad \forall t \geq 0,
\]
and given a value $r \leq r(\omega)$, one can obtain an updated upper bound $U(m, \omega, r) : [0, +\infty[ \to [0, +\infty[$ such that
\[
	\|S(t)\| \leq U(m, \omega, r)(t) \leq m(t), \quad \forall t \geq 0.
\]
}
\end{quote}

\begin{example}
Theorem \ref{th3.2} below, taken from \cite{W}, is equivalent to saying that one may take with $\omega=0$, $r = r(0)$ and $m \equiv 1$
\[
	U(1, 0, r)(t) = \begin{cases} 1, & 0 \leq t \leq \frac{\pi}{2r}, \\ \exp(\frac{\pi}{2} - rt), & t > \frac{\pi}{2r}.\end{cases}
\] 
That is, if $\|S(t)\| \leq 1$ for $t \geq 0$ and if the generator of $S(t)$ satisfies $r(0) = r$ as defined in Definition \ref{def:r}, then $\|S(t)\| \leq U(1, 0, r)(t)$ for all $t \geq 0$.

In Section \ref{Section3} we show that this upper bound is optimal for $0 \leq t \leq \frac{\pi}{2r}$. However, for any $t_0 > \frac{\pi}{2r}$ we do not know of an example of a semigroup with $S(t) = \exp(tA)$ with $-A$ an $m$-accretive operator on a Hilbert space such that $\|S(t_0)\| = \exp(\frac{\pi}{2} - rt_0)$.
\end{example}

\subsection{The main theorem in \cite{HelSj} and discussion on connected results}
We recall from \cite[Theorem 1.7]{HeSj21} the following improvement of the main result in \cite{HelSj}.

\begin{theorem}\label{int2}
Suppose $\omega \in \Bbb{R}$ is such that $r(\omega)$ defined in \eqref{def:r} is strictly positive.
Let $m(t): [0, +\infty[ \to ]0, +\infty[$ be a continuous positive function such that 
\begin{equation}\label{eq:h1}
\|S(t)\| \leq m(t) \mbox{ for all } t \geq 0\,.
\end{equation}
Then for all $t,a,b>0$ such that $t \geq  a+b$,
\begin{equation}\label{int.4str}  \Vert S(t)\Vert \le \frac{e^{\omega t - r(\omega) (t-a-b)} }{r(\omega )\Vert
    \frac{1}{m}\Vert_{e^{-\omega \cdot }L^2(]0,a[)}\Vert
    \frac{1}{m}\Vert_{e^{-\omega \cdot }L^2(]0,b[)}},
    \end{equation}
where for $c > 0$
\[
	\|f\|_{e^{-\omega\cdot}L^2(]0, c[)}^2 = \int_0^c |f(t)|^2 e^{2\omega t}\,\dd t.
\]
\end{theorem}

We now present some applications of this theorem with comparisons with the existing literature.

\begin{enumerate}
\item
The proof of this theorem (together with applications) was first presented in \cite{HelSj} and later published in the books \cite{He1,Sj2}. The advantage of this result, compared with prior works, is that all the constants are explicit. The version given in \cite{HelSj} has only the weaker statement that, for $t \geq a+b$,
\begin{equation}\label{int.4}  \Vert S(t)\Vert \le \frac{e^{\omega t } }{r(\omega )\Vert
    \frac{1}{m}\Vert_{e^{-\omega \cdot }L^2(]0,a[)}\Vert
    \frac{1}{m}\Vert_{e^{-\omega \cdot }L^2(]0,b[)}}\,.
\end{equation}
\item As observed in \cite{HelSj}, one can vary $\omega > \omega_1$ when some information on the behavior of $r(\omega)$ as $\omega \rightarrow \omega_1$ is given. For example, we get from \eqref{int.4} by considering $\omega=\omega_1+ \frac 1 t$, the inequality for $t$ large enough:
\begin{equation}\label{int.4bis}  \Vert S(t)\Vert \le e \frac{e^{\omega_1 t}} {r(\omega_1+1/t )\Vert
    \frac{1}{m}\Vert_{e^{-\omega \cdot }L^2(]0,a[)}\Vert
    \frac{1}{m}\Vert_{e^{-\omega \cdot }L^2(]0,b[)}}\,.
    \end{equation}
Various particular cases have been considered in the literature \cite{R,Wa,EZ,Ei},  often with worse constants:
    \begin{itemize}
    \item Assuming for example that $m(t)=M$ on $[0,a]$, $a = b > 0$ and $\omega_1=0$, we get, as stated in \cite{RV} for $t\geq 2a$
    \begin{equation}\label{int.4ter}  \Vert S(t)\Vert \le \frac{e M^2}{a}  \frac{1} {r(1/t)}\,.
    \end{equation}
    \item A particular attention is given to the case of semi-group satisfying the so-called $\alpha$-Kreiss-condition. This corresponds to the case when $\omega_1\leq 0$ and 
    $$
    r(\omega) \geq \frac{1}{C_\alpha} \omega^{\alpha}\,,\, \omega \in ]0,\hat \omega_0]\,.
    $$
    When $\alpha =1$, the estimates can be improved using a Ces\'aro averaging method (see \cite{Ar,Bou}) and one can gain a factor $\frac{1}{\sqrt{\log t}}$ for $t$ large. This improvement does not seem accessible using the techniques of \cite{HelSj} or \cite{HeSj21}.
    
As known (see \cite{Ei}, Remark 1.22, p.~90)) the case when $\alpha < 1$ implies that the semi-group is exponentially stable, meaning that there exists $L$ and $\omega' <0$ such that the semi-group satisfies $P(L,\omega')$. This is a consequence of \eqref{eq:monot} which implies that $\omega_1<0$ with $\omega_1$ from \eqref{eq:omega_1}.
\item Let us give three consequences of \eqref{int.4str} which are better estimates than what one could obtain from \eqref{int.4}.
     \begin{enumerate}
     \item If for some $\omega >0$, we have $\omega < r(\omega)$, the semi-group is exponentially stable and we can measure through \eqref{int.4str} its asymptotic decay as $t \rightarrow +\infty$.
    \item If for some $\omega > 0$, we have $\omega\leq r(\omega)$, the semi-group is bounded. This is a weak form of the Hille-Yosida Theorem but under much weaker
     assumptions. 
    \item If for some $\omega_0 >0$, $\ell >1$ and $C>0$, we have
    $$
   0> r(\omega) -\omega \geq - C \omega^\ell\,,\, \forall \omega \in ]0,\omega_0]\,,
    $$
    then there exists $\widehat C >0$ such that
    $$
    || S(t) || \leq \widehat C\, (1 + t^{1/\ell})\,.
    $$
      For this statement, we just apply \eqref{int.4str} with $\omega = t^{-1/\ell}$ and $t\geq t_0$.
    \end{enumerate}
    
     \end{itemize} 
     \item The Banach case (with some $L^p$ version) has been considered  in \cite{LY} and later in \cite{Hel22}. In this case $\frac{1}{r(\omega)}$ has to be replaced by the norm in $\mathcal L (L^p_\omega(\mathbb R_+;\mathcal H))$
     of  the operator $\mathcal K^+$ defined by
$$
(\mathcal K ^+ u)(t)=\int_0^t   S(t-s) u(s) ds\,,\, \mbox{ for } t\geq 0\,.
$$
Here $L^p_\omega(\mathbb R_+;\mathcal H) := e^{\omega \cdot} L^p(\mathbb R_+;\mathcal H)$.
\item The paper \cite{Wa} contains many results and is posterior to \cite{Ar} and \cite{BT}. In the spirit of  its Proposition 3.5 (who gives a nice elegant proof), we can get
 with a better constant and with the same proof as for Theorem \ref{int2} the following extension:
\begin{theorem}\label{int2Wa}~\\
Let $C \in \mathcal L (D(A),\mathcal H)$ be such that $S(t)$ commutes with $C$ for all $t\geq 0$. 
Suppose $\omega \in \Bbb{R}$ is such that
$$
r_C(\omega)^{-1}:= \sup_{\Re \lambda \geq \omega} || C R(\lambda,A)|| < +\infty \,.
$$
Let $m(t): [0, +\infty[ \to ]0, +\infty[$ be a continuous positive function such that 
\begin{equation}\label{eq:h1new}
\|S(t)\| \leq m(t) \mbox{ for all } t \geq 0\,.
\end{equation}
Then the operator extends to a bounded operator on $\mathcal H$ for all $t>0$ and 
 for all $t,a,b>0$ such that $t \geq  a+b$, we have
\begin{equation}\label{int.4new}  \Vert C  S(t)\Vert \le \frac{e^{\omega t}} {r_C(\omega )\Vert
    \frac{1}{m}\Vert_{e^{-\omega \cdot }L^2(]0,a[)}\Vert
    \frac{1}{m}\Vert_{e^{-\omega \cdot }L^2(]0,b[)}}\,.
    \end{equation}
\end{theorem}
The particular case when $C$ is a projector was considered in \cite[Theorem 1.6]{HelSj}.
\end{enumerate}
Finally note that another application of Theorem \ref{int2} will appear in Section~\ref{Section3}.

\subsection{Wei's theorem and the generalizations in \cite{HeSj21}}
\subsubsection{Wei's theorem} In \cite{W}, Dongyi Wei, motivated by  \cite{HelSj},  has proved the following theorem:
\begin{theorem}\label{th3.2}
 Let $H=-A$ be an $m$-accretive operator in a Hilbert space $\mathcal H$. Then we have
\begin{equation}\label{eq:w}
||S(t) || \leq e^{- r(0)t +\frac \pi 2}\,,\, \forall t\geq 0\,.
\end{equation}
\end{theorem}
 Of course this result is only interesting for $t > \frac{\pi}{2r(0)}$. Note that it cannot be obtained directly from Theorem \ref{th3.2} which leads to a larger constant (see \cite{HeSj21} for this discussion).\\ Motivated by this theorem the two first authors generalize Theorem \ref{int2} as we now explain.

\subsubsection{Riccati equation and definition of $a^*, b^*$}\label{sss:astar_bstar}
Define $F$ on $\mathbb R_+\times \mathbb R$ by
 $$F(x,y):= - (x^2 + 2x y +1)\,.$$
For $\mu:[0, +\infty[ \to \Bbb{R}$ absolutely continuous, we pose $\Psi:]0, T[ \to \Bbb{R}$ the maximal solution to
\begin{equation}\label{defRic}
	\begin{cases}
	\Psi'(b) = F(\Psi(b),\mu(b))
	\\
	\lim_{b \to 0^+} \Psi(b) = +\infty.
	\end{cases}
\end{equation}
We then define
\begin{equation}\label{defa_0^*}
	b^*(\mu) := \inf \{b> 0\::\: \Psi(b) = 1\},
\end{equation}
with the usual understanding that if $\Psi(b)$ is never equal to $1$ then $b^*(\mu) = +\infty$.
If needed we write $\Psi(b;\mu)$ when we want to mention the dependence on $\mu$\,.

Equivalently, if
\begin{equation}\label{defPhi}
	\begin{cases}
	\Phi'(b) = - F(\Phi(b),\mu(b))
	\\
	\Phi(0) = 0,
	\end{cases}
\end{equation}
then it can be proven that $\Phi(b) = 1/\Psi(b)$ for all $b \in ]0, b^*(\mu)]$ and
\begin{equation}
	b^*(\mu) = \inf \{b > 0 \::\: \Phi(b) = 1\}.
\end{equation}
The proofs of these results and others are given in \cite{HeSj21} (Section 3).

In Theorem \ref{thm:Riccati} which follows, one uses a rescaled version of $b^*$. When
\begin{equation}\label{eq:def_mu}
	\mu(t) = \mu(t; m, \omega, r) = \frac{1}{r}\left(\frac{m'(t)}{m(t)} - \omega\right),
\end{equation}
and
\begin{equation}\label{eq:def_phi}
	\begin{cases}
	\phi'(t) = r(\phi(t)^2 + 2\mu(t)\phi(t) + 1),
	\\
	\phi(0) = 0,
	\end{cases}
\end{equation}
let
\begin{equation}\label{eq:def_astar}
	a^* = a^*(m, \omega, r) = \inf\{t \geq 0 \::\: \phi(t) = 1\}.
\end{equation}
One may check that 
\begin{equation}
	a^* = \frac{1}{r} b^*(\mu).
\end{equation}

\subsubsection{Extension of Wei's Theorem}
We will study applications of the following theorem, reformulated from \cite[Theorem 1.10]{HeSj21}.

\begin{theorem}\label{thm:Riccati}
Let $S(t)$ be a one-parameter strongly continuous semigroup acting on a Hilbert space $\mathcal{H}$ with generator $A$. Suppose that $m:[0, +\infty[ \to ]0, +\infty[$ is a positive continuous function 
with piecewise continuous derivative such that \eqref{eq:h1} holds.

Let $\omega \in \Bbb{R}$ be such that $r(\omega)$ from Definition \ref{def:r} is positive and let
\begin{equation}\label{eq:condromega}
r \in ]0, r(\omega)].
\end{equation}
Let $a^* = a^*(m, \omega, r)$ be as in \eqref{eq:def_astar}. Then, for every $t \geq 2a^*$,
\begin{equation}\label{eq:h5}
	\|S(t)\| \leq \ee^{(\omega - r)(t - 2a^*)}m(a^*)^2.
\end{equation}
\end{theorem}

One way to write the result of this theorem is that we have the updated upper bound
\[
	\|S(t)\| \leq U(m, \omega, r)(t)
\]
where
\begin{equation}\label{eq:thm_Riccati_U}
	U(m, \omega, r)(t) = 
	\begin{cases} 
	m(t), & 0 \leq t \leq 2a^*,
	\\
	\min\{m(t), \ee^{(\omega - r)(t-2a^*)}m(a^*)^2\}, & t > 2a^*.
	\end{cases}
\end{equation}

\subsection{Goal and organization of the paper}

The goal of the present work is to explore the optimality of this result and possible improvements which could be obtained by iteration of the theorem or its application with different pairs $(\omega,r)$ with $r \leq r(\omega)$.

In Section 2, we discuss how the result of Theorem \ref{thm:Riccati} depends on the parameters $r$, $\omega$, and $(\log m)'$. In Section 3, we discuss the example of a shift on a bounded interval, for which the bound of Theorem \ref{th3.2} is optimal up to a finite time. In Section 4, we discuss improvements of upper bounds which come from the semigroup property $\|S(t_1 + t_2)\| \leq \|S(t_1)\| \|S(t_2)\|$. Finally, in Section 5, we consider iterations of Theorem \ref{thm:Riccati} which give sequences of upper bounds with piecewise affine logarithms.

\section{Dependence of Theorem \ref{thm:Riccati} on parameters}

To better understand solutions to \eqref{eq:def_phi}, we study how these solutions depend on the parameters $r > 0$, $\omega \in \Bbb{R}$, and the function $\mu$. We regard each as an independent variable, even though for instance $r$ depends on $\omega$ in the applications we consider.

\subsection{Monotonicity with respect to $r,\omega$}\label{ss:monot_romega}

To begin, we suppose that the function $\mu$ is fixed. To remove the dependence on $r$ of $\mu$ in \eqref{eq:def_mu}, we define
\begin{equation}\label{eq:def_mu1}
	\mu_1 = r\mu = \frac{m'(t)}{m(t)} - \omega.
\end{equation}
We study the dependence of $a^*$ in \eqref{eq:def_astar} on $r$ and on $\omega$ separately, regarded $r$ and $\omega$ as independent variables.

\begin{proposition}\label{prop:astardecr}
Let $a^* = a^*(m, \omega, r)$ be defined as in \eqref{eq:def_astar}. When $m$ and $\omega$ are fixed, $a^*$ is a decreasing function of $r$. When $m$ and $r$ are fixed, $a^*$ is an increasing function of $\omega$.
\end{proposition}

\begin{proof}
We suppose that $m$ is fixed throughout. We begin by showing that $\phi$ is an increasing function of $r$ for $\omega$ fixed, and we will therefore write 
\[
	\phi = \phi(t, r), \quad a^* = a^*(r)
\]
where $a^*(r)$ the first solution to $\phi(a^*(r), r) = 1$ as in \eqref{eq:def_astar}. We claim that 
\begin{equation}\label{1.addS2}
\partial _r\phi(t, \omega, r) > 0, \quad \forall t \in ]0,a^*(\omega, r)].
\end{equation}
To see this, we differentiate with respect to $r$ the equation satisfied by $\phi$
\begin{equation}\label{1.addS1bis}
\partial_t \phi (t, r) = G (\phi(t, r), \mu_1(t), r)
\end{equation}
with $\mu_1$ from \eqref{eq:def_mu1} and
\begin{equation}
G(x,y,r) =  r x^2 + r + 2 x  y.
\end{equation}
We get, with 
$$
g(t,r) := (\partial_r\phi )(t, r)\,,
$$
a linear ordinary differential equation of order one in the $t$ variable for $g$ which reads
\begin{subequations}\label{3.addS2new}
\begin{equation}\partial _t  g(t,r)  -(\partial _x G)(\phi,\mu_1,r)
g=(\partial _r G)(\phi,\mu_1,r ) = \phi^2 + 1 > 0\,,
\end{equation}
on $]0,a^*]$, with initial condition at $0$
\begin{equation}
g(0,r)=0\,.
\end{equation}
\end{subequations}
We introduce
$$
h(t) = e^{\theta(t)} g(t)\,,
$$
with 
$$
\theta'(t)= (\partial _x G)(\phi,\mu_1,r)\,,
$$
which satisfies
\begin{equation}\label{eq:dependence_r_h}
h'(t) = e^{\theta}(\partial _r G)(\phi,\mu_1,r )>0\,,\, h(0)=0\,.
\end{equation}
This implies $h> 0$ and coming back to $g$ and $\phi$
we have proven that
$\phi(t, \omega, r)$  is an increasing function of $r$, so $a^*(r)$ is decreasing.

Next, we suppose that $r$ is fixed and $\omega$ varies (while $m$ remains fixed). To emphasize this choice, we will now write
\[
	\phi = \phi(t, \omega), \quad a^* = a^*(\omega).
\]
Notice that, with $\mu_1$ from \eqref{eq:def_mu1},
\[
	\partial_\omega \mu_1 = -1.
\]
Therefore for $t \in ]0, a^*(\omega)]$
\begin{subequations}
\begin{equation}
	\partial_t(\partial_\omega \phi) - (\partial_x G)(\phi, \mu_1, r)\partial_\omega \phi 
	= - (\partial_y G)(\phi, \mu_1, r) 
	= - 2\phi < 0.
\end{equation}
We also have
\begin{equation}
	\partial_\omega \phi(0, \omega) = 0.
\end{equation}
\end{subequations}

The same argument which gave \eqref{eq:dependence_r_h} gives this time that
\begin{equation}\label{5.addS2bis}
\partial _\omega \phi < 0 \mbox{ on } ]0,a^*(r)]\,.
\end{equation}
Therefore $\phi$ is a decreasing function of $\omega$ and $a^*$ is an increasing function of $\omega$, completing the proof.
\end{proof}

\subsection{Monotonicity with respect to $\mu$}

We now suppose that there is some parameter $\theta$ varying in an interval $J$ such that $\mu(t) = \mu(t, \theta)$ varies smoothly in $\theta$. In this subsection, we therefore write
\[
	\phi(t,\theta)
\]
for the solution to \eqref{eq:def_phi} to emphasize this dependence. We also write $a^*(\theta)$ for the first solution to $\phi(a^*(\theta), \theta) = 1$ as in \eqref{eq:def_astar}.

Recall that $0 < \phi(t,\theta)  \le 1$ for $t \in [0, a^*(\theta)]$, and let 
\[
	g(t,\theta) = -\log \phi(t, \theta) \ge 0, \quad t \in [0, a^*(\theta)].
\]
Then (\ref{eq:def_phi}) implies that
\begin{equation}\label{monom.2}
\partial_tg=-2(\mu +r\cosh g).
\end{equation}
Differentiating (\ref{monom.2}) with respect to $\theta $, we get
\begin{equation}\label{monom.3}
\partial _t\partial_\theta g+2r(\sinh g)\partial_\theta
g=-2\partial _\theta \mu . 
\end{equation}
Here, we claim that
\begin{equation}\label{monom.4}
\lim_{t \rightarrow 0^+}\partial_\theta g(t,\theta )=0\,.
\end{equation}
To see this, we come back to  \eqref{monom.2} for $\phi$ 
which we write in the form
\begin{equation}\label{monom.1bis}
\partial _t\phi  =2\mu \phi + r  (1 + \phi^2) \,,\, \phi (0)=0\,,
\end{equation}
and get the Taylor expansion
\begin{equation}\label{monom.5}
\phi(t,\theta )= rt+{\mathcal O}(t^2),\ t\to 0.
\end{equation}
We get
\begin{equation}\label{monom.6a}
g(t,\theta )=\log \frac{1}{rt}+\log (1+\mathcal{O}(t))=\log \frac{1}{rt}+\mathcal{O}(t),
\end{equation}
\begin{equation}\label{monom.6b}
\partial_\theta g(t,\theta )= \mathcal{O}(t),
\end{equation}
and (\ref{monom.4}) follows.

From (\ref{monom.3}) and (\ref{monom.4}), we get
\begin{equation}\label{monom.7}
\partial _\theta g(t,\theta )=-2\int_0^t
e^{-2r\int_s^t\mathrm{sinh\,}g(\sigma )d\sigma } \partial _\theta \mu
(s,\theta )\, ds.
\end{equation}

Recall that $a^*=a^*(\theta )$ satisfies  $\phi (a^*(\theta ),\theta
)=1$, which implies 
\begin{equation}\label{monom.9}
g(a^*(\theta ), \theta)=0.
\end{equation} 
Assume that $a^*(\theta )$ satisfies
\begin{equation}\label{monom.10}
\mu (a^*(\theta ),\theta )
+ r > 0 \,,
\end{equation}
at some point $\theta$. By \eqref{monom.2} this implies that $a^*(\theta)$ is differentiable and $g(a^*(\theta), \theta)$ has a nondegenerate zero in \eqref{monom.9}. Differentiating (\ref{monom.9}) in a suitably small neighborhood of $\theta$, we get
$$
(\partial _tg)(a^*(\theta ),\theta )\partial _\theta a^*(\theta
)+(\partial _\theta g)(a^*(\theta ),\theta )=0,
$$
i.e.\
\begin{equation}\label{monom.11}
\partial _\theta a^*(\theta )=-\frac{\partial _\theta g}{\partial
  _tg}(a^*(\theta ),\theta ). 
\end{equation}

\begin{proposition}\label{monom1}
Assume (\ref{monom.10}) and
\begin{equation}\label{monom.12}
\partial _\theta \mu (t,\theta )\ge 0,\ \ 0\le t \le a^*(\theta ).
\end{equation}
Then $\partial _\theta a^*(\theta )\le 0$.
\end{proposition} 
\begin{proof}
From (\ref{monom.7}), (\ref{monom.12}) we
see that $\partial _\theta g(a^*(\theta ),\theta )\le 0$ and it
suffices to combine this, \eqref{monom.2}, and (\ref{monom.10}) with
(\ref{monom.11}).
\end{proof}

\begin{remark}
If we assume instead of (\ref{monom.12}) that $\partial _\theta
\mu \le 0$ on $[0,a^*(\theta )]$, then $\partial _\theta a^*(\theta
)\ge 0$. 
\end{remark}

\begin{remark}
Recall that $\mu_1 (t,\theta ) + \omega =\partial _t\log m(t,\theta
)$. Assuming that $m(0,\theta )=1$, we get
$\log m(t,\theta )=\int_0^t (\mu_1 (s,\theta ) + \omega) ds$, so the assumption that
$\partial _\theta \mu (s,\theta )\ge 0$ on $[0,a^*(\theta )]$ implies
that $\partial _\theta \log m(t,\theta )\ge 0$ on the same interval and
hence that $\partial _\theta m(t,\theta )\ge 0$. But the converse is not necessarily true.
\end{remark}

\section[The differentiation operator]{Analysis of the differentiation
  operator on an interval.}\label{Section3}

The starting point is a paragraph in \cite[Chapter~14]{He1} presenting a toy model described by Embree-Trefethen \cite[Chapter~15]{TrEm}. The goal is to prove that in this case the Wei constant $e^{\pi/2}$ in \eqref{eq:w} is optimal in the sense that
\[
	\sup\{\ee^{r(0)t}\|\exp(tA)\|\::\: -A\textnormal{ is }m\textnormal{-accretive}, t \geq 0\} = \ee^{\pi/2}.
\]

We consider the operator $A$ defined on $L^2(]0,1[)$ by
\begin{subequations}\label{eq:3.1}
\begin{equation}
D(A)=\{u \in H^1(]0,1[)\,,\, u(1)=0\}\,,
\end{equation}
and
\begin{equation}
Au = u'\,,\, \forall u\in D(A)\,.
\end{equation}
\end{subequations}
This is clearly a closed operator with dense domain.

The adjoint of $A$ is defined on $L^2(]0,1[)$
 by
$$
D(A^*)=\{u \in H^1(]0,1[)\,,\, u(0)=0\}\,,
$$
and
$$
A^*u = - u'\,,\, \forall u\in D(A^*)\,.
$$

\begin{lemma}~\\ With $A$ as above, 
$\sigma(A)=\emptyset$ and $A$ has compact resolvent.
\end{lemma}

\begin{proof}
First we can observe that $(A-z)$ is injective on $D(A)$ for any $z\in
\mathbb C$. This is simply the observation that $u\in {\rm Ker } (A-z)$ should satisfy $$- u'(t)= z\,  u(t) \mbox{ and } u(1)=0\,.$$ One also  easily verifies that, for any $z\in \mathbb C$,
the
 inverse is given by
\begin{equation}\label{inverse}
[(z-A)^{-1} f](x)  = -  \int_x^1 \exp z(x-s)\; f(s)\, ds\,.
\end{equation}
It is also clear that this operator is compact (for example because its distribution kernel is in $L^2(]0,1[\times ]0,1[)$).
\end{proof}

We recall that 
\begin{equation}
	\psi(z):=||(A-z)^{-1}||
\end{equation}
is subharmonic. Observing that, for any $\alpha \in \mathbb R$, the map 
\[
	u\mapsto U_\alpha u :=\exp (i\alpha x) \, u
\]
is a unitary
 transform on $L^2(]0,1[)$, which maps $D(A)$ onto $D(A)$ and such that  $U_\alpha^{-1} A U_\alpha = A + i \alpha$\,,  we deduce that $\psi$ depends only on $\Re z$. 

For $u\in D(A)$ and $z \in \Bbb{C}$,
$$
- \Re \langle (A-z) u\,,\, u\rangle = \Re z \|u\|^2 + \frac 12  |u(0)|^2\geq \Re z
\|u\|^2\,.
$$
In particular $- A$ is accretive and satisfies the assumptions of Theorem \ref{th3.2} of D.~Wei.

In order to apply Theorem \ref{th3.2}, we have to compute $r(0)=1/\psi(0)$. Hence we have to compute $||A^{-1} ||$. In our case, we get that $r(0)$ is the square root of the smallest eigenvalue of $A^*A$. It is easy to show that the domain of $A^*A$ is given by
\begin{subequations}
\begin{equation}
D(A^* A)= \{u\in H^2([0,1]), u'(0)=u(1)=0\}\,,
\end{equation}
and that  for $u\in  D(A^*A)$
\begin{equation}
A^*A u = - u''\,.
\end{equation}
\end{subequations}
The lowest eigenvalue is $\pi^2/4$ with corresponding eigenspace generated by  $u_0(x)= \cos (\pi x/2)\,$. So finally we have 
$$
r(0) =\pi/2\,,
$$ 
and Wei's theorem gives
\begin{equation}
||S(t)|| \leq \exp\left( \pi/2 (1-t)\right)\,.
\end{equation}
On the other hand, one can directly compute the norm of $S(t) $. We have indeed for $u\in L^2(]0,1[)$:
$$
(S(t) u) (x) = \tilde u (x+t)\,,
$$
where $\tilde u$ is the extension of $u$ by $0$ on $]1,+\infty[$. For $t>1$, one immediately sees that 
$$
S(t)= 0\,.
$$
For $t < 1$, one gets
$$
||S(t)||=1\,.
$$
We have indeed
$$
1 \geq || S(t)|| \geq || S(t) \phi ||_{L^2} \,,
$$
for any normalized $L^2$ function on $]0,1[$. If we choose $\phi_t$ such that $||\phi_t||=1$ and ${\rm supp} (\phi_t) \subset (t,1) $, we immediately get
$$
|| S(t) \phi||_{L^2}= ||\tilde \phi_t(\cdot +t)||= 1\,.
$$
We now show that Wei's constant is optimal. Suppose that there exists $C<1$ such that 
\begin{equation}
||S(t)|| \leq C \exp\big( \pi/2 (1-t)\big)\,.
\end{equation}
For $t <1$, this implies
\begin{equation}
1 \leq C \exp\big( \pi/2 (1-t)\big)\,.
\end{equation}
We get a contradiction as $t \rightarrow 1$ (with $t<1$).

\begin{remark}
As observed in \cite{TrEm}, one can discretize the preceding problem by considering, for $n\in
\mathbb N^*$,  the matrix 
$ A_n = n A_1$ with $A_1 =I + J$
 where $J$ is the $n\times n$ matrix such that
$J_{i,j} = \delta_{i+1,j}$. One can observe that the spectrum of
 $A_n$ 
 is $\{n\}$. 
\end{remark}

An interesting theorem related to the present study is \cite[Theorem~15.6]{TrEm}.
\begin{theorem}\label{thm:sg_vanishes}
Let $A$ be a closed linear operator generating a $C_0$ semigroup. For any $\tau >0$, the following properties are equivalent:
\begin{enumerate}[(a)]
\item\label{it:vanishes_1} $e^{\tau A} =0$\,.
\item\label{it:vanishes_2}  $\sigma (A)=\emptyset $ and  there exists $C>0$ and $\omega_0 < 0$ such that, for $\omega \in ]-\infty,\omega_0]$ 
\begin{equation}\label{eq:abcd}
 \frac{1}{r(\omega)} \leq C e^{ -\tau \omega}\,.
\end{equation}
\end{enumerate}
\end{theorem}
\begin{proof}
The proof that \ref{it:vanishes_1} implies \ref{it:vanishes_2} is a consequence of the formula $$(A-z)^{-1} = \int_0^{\tau} e^{- tz } S(t) dt\,,$$ together with the Banach-Steinhaus theorem.

The proof that  \ref{it:vanishes_2}  implies  \ref{it:vanishes_1} is an easy application of Theorem \ref{int2}. By the Banach-Steinhaus Theorem and the semi-group theory we can take, for some $M>0$ and $\omega_0 \geq 0$ 
$$
m(t)=M \exp \omega_0 t\,.
$$
(The accretive case corresponds to $M=1$ and $\omega_0=0$.)

We apply the theorem in the limiting case when $a=b= t/2$ and \eqref{int.4} gives when $\omega < \omega_0$
 \begin{equation}\label{eq:gp1aa}
 || S(t)|| \leq 2 M^2 (\omega_0 - \omega) \frac {e^{\omega t} }{r(\omega)} \, (1 -e^{(\omega -\omega_0) t})^{-1} \,.
    \end{equation}
By \eqref{eq:abcd}, for all $\omega \leq \omega_0$
\[
	\frac{\ee^{\omega t}}{r(\omega)} \leq Ce^{\omega(t-\tau)}.
\]
When $t >\tau$, in the limit $\omega \rightarrow -\infty$ the estimate \eqref{eq:gp1aa} gives $S(t) = 0$ as claimed.
\end{proof}
 
When applied to the differential operator $A$ introduced in \eqref{eq:3.1}, the estimate \eqref{eq:abcd} is proven with $\tau=1$ (see \cite{TrEm} or \cite[Chapter 14, (14.1.3)]{He1}).
We propose below an alternative approach to the control of $||(A-z)^{-1}||$ for our differential operator $A$.

We begin by introducing the function $\nu$ such that, for all $\omega \in \Bbb{R}$,
\[
	\|(\omega-A)^{-1}\| = r(\omega) = \sqrt{\omega^2 + \nu(\omega)^2}\,.
\]

\begin{proposition}\label{prop:diffop_nu}
There exists a unique continuous function $\nu(z)$ on $\Bbb{R}$ with values in $\ii]0, +\infty[$ for $z < -1$ and in $[0, \pi[$ when $z \geq -1$ such that
\[
	-\nu(z) \cot \nu(z) = z, \quad \forall z \in \Bbb{R} \backslash \{-1\}.
\]
The functions $\nu(z)^2$ and $z^2 + \nu(z)^2$ are increasing, and as $z \to -\infty$,
\begin{equation}\label{eq:diffop_r_asymptotic}
	z^2 + \nu(z)^2 = 4z^2 \ee^{-2|z|}(1 + \BigO(\ee^{-2|z|}).
\end{equation}
\end{proposition}

\begin{proof}
The function
\[
	f(\nu) = -\nu \cot \nu, \quad \nu \in ]0, \pi[
\]
is increasing, since its derivative is $f'(\nu) = (2\sin^2\nu)^{-1}(2\nu - \sin(2\nu))$, with limits
\[
	\lim_{\nu \to 0^+} f(\nu) = -1, \quad \lim_{\nu \to \pi^-} f(\nu) = +\infty.
\]
Similarly,
\[
	g(\eta) = f(\ii \eta) = -\eta \coth \eta, \quad \eta \in ]0, +\infty[
\]
is decreasing with
\[
	\lim_{\eta \to 0^+} g(\eta) = -1, \quad \lim_{\eta \to +\infty} g(\eta) = -\infty.
\]

These functions together allow us to define an implicit function $\nu : \Bbb{R} \to \ii]0, +\infty[ \cup [0, \pi[$ such that $\nu(-1) = 0$ and
\[
	f(\nu(z)) = z, \quad z \in \Bbb{R}\backslash \{-1\}.
\]
If $z \in ]-1, +\infty[$ there is a unique $\nu(z) \in ]0, \pi[$ satisfying $f(\nu(z)) = z$. If $z < -1$, there is a unique $\eta(z) \in ]0, +\infty[$ satisfying $g(\eta(z)) = z$, and we let $\nu(z) = \ii\eta(z).$ Hence, we have $g(\eta(z)) = f(\nu(z)) = z$. The value of $\nu(-1)$ ensures that $\nu(z)$ is continuous on all of $\Bbb{R}$.

Because $f:]0, \pi[ \to ]-1, +\infty[$ is increasing, $\nu(z)$ is an increasing function from $]-1, +\infty[$ to $]0, \pi[$, and $\nu(z)^2$ is likewise increasing on $]-1, \infty[$. For $z \in ]-\infty, -1[$, $\eta(z) = -\ii \nu(z)$ is decreasing from $]-\infty, -1[$ to $]0, +\infty[$, so $\nu(z)^2 = -\eta(z)^2$ is increasing. Since $\nu$ is continuous, $\nu(z)^2$ is increasing on $\Bbb{R}$.

As for $z^2 + \nu(z)^2$, if $z > -1$ then
\[
	z^2 + \nu^2 = (-\nu \cot \nu)^2 + \nu^2 = \left(\frac{\nu}{\sin\nu}\right)^2.
\]
The function $\nu/\sin\nu$ is increasing and positive on $]0, \pi[$, as can be seen from its derivative
\[
	\left(\frac{\nu}{\sin\nu}\right)' = \frac{1}{\sin \nu^2}\,(\sin \nu - \nu\cos \nu) > 0.
\]
Therefore $z^2 + \nu(z)^2$ is increasing on $]-1, +\infty[$, as an increasing function of an increasing function. If $z < -1$, then $\nu(z)^2 = -\eta(z)^2$ with $\eta(z)\in]0, +\infty[$ a decreasing function of $z$ satisfying $z = -\eta\coth \eta$. We compute similarly
\[
	z^2 + \nu^2 = \eta^2\coth^2\eta - \eta^2 = \left(\frac{\eta}{\sinh\eta}\right)^2,
\]
with $\eta / \sinh \eta$ on $]0, +\infty[$ a positive function which is decreasing (because its derivative is $(\tanh\eta - \eta)\coth^2\eta$). Therefore $z^2 + \nu(z)^2$ is increasing on $]-\infty, -1[$ as well as on $]-1, +\infty[$, which extends by continuity to all of $\Bbb{R}$.

As for the asymptotic behavior of $\nu(z)^2 + z^2$ as $z \to -\infty$, using that
\[
	\coth\eta = 1 + 2\ee^{-2\eta} + \BigO(\ee^{-4\eta})
\] 
as $\eta \to +\infty$, from $z = -\eta \coth \eta$ we obtain
\[
	\eta = -z\left(1 - 2\ee^{-2|z|} + \BigO(\ee^{-4|z|})\right).
\]
The claim \eqref{eq:diffop_r_asymptotic} immediately follows.
\end{proof}

\begin{proposition}\label{prop:diffop_svd}
Let $A$ be defined as in \eqref{eq:3.1}. Let $\nu(z)$ be defined as in Proposition  \ref{prop:diffop_nu}.
Then for $z = z_1 + \ii z_2$ with $z_1, z_2 \in \Bbb{R}$,
\[
	\|(A - z)^{-1}\| = (z_1^2 + \nu(z_1)^2)^{-1/2}.
\]
and for $\omega \in \Bbb{R}$, with $r(\omega)$ defined in Definition \ref{def:r}, we have
\begin{equation}\label{eq:diffop_r}
	r(\omega) = \sqrt{\omega^2 + \nu(\omega)^2}.
\end{equation}
\end{proposition}

\begin{figure}
\centering
\includegraphics[width = .5\textwidth]{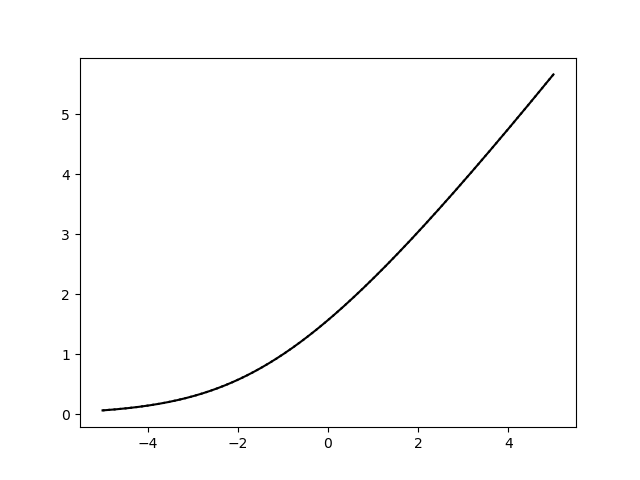}
\caption{Graph of $r(\omega)$ for $A$ in \eqref{eq:3.1}.}
\end{figure}

\begin{proof}
As in the case $z=0$, the proof is based on the property that, for $z\in \mathbb R$,  $1/ ||(A-z)^{-1}||$ is the square root of the smallest eigenvalue of the operator
\begin{equation}
B(z) := (A^* -z) (A-z) 
\end{equation}
whose domain reads
\begin{equation}
D(B(z)) =\{u \in H^2(0,1), u(1)=0, u'(0) - z u(0)=0\}\,.
\end{equation}
and is a realization of $-\frac{d^2}{dx^2}+ z^2$ on this domain. At the end we will be interested in the square root of the lowest eigenvalue of $B(z)$.

We first analyze the spectrum of 
$$
C(z):= B(z)-z^2 = -\frac{\dd^2}{\dd x^2}.
$$
It is rather standard to  determine the lowest eigenvalue as a function of $z\in \mathbb R$. If $(C(z) - \nu^2) \phi_\nu(x) = 0$ for $\nu > 0$ and if $\phi_\nu(1) = 0$, then up to constants
\begin{equation}\label{eq:diffop_svd_sin}
\phi_\nu(x)=\sin \nu (1-x)\,.
\end{equation}
Here $\nu$ is determined by the Robin condition at $0$:
$$
\nu  \cos \nu = - z \sin \nu\,.
$$
If $\nu (z)$ is a solution of this equation, the corresponding eigenvalue will be $\nu(z)^2$. We choose $\nu(z)$ such that this eigenvalue is minimal. By the Sturm-Liouville property $\phi_\nu(x)$ does not vanish in $]0,1[$, and therefore $\nu < \pi$. Since $\nu \cot \nu$ decreases from $1$ to $-\infty$ for $\nu \in [0, \pi[$, we conclude that there is a candidate $\phi_\nu$ for the eigenfunction of $C(z)$  with smallest eigenvalue if and only if $z \in ]-1, \infty[$.

If $C(z)\phi_0(x) = 0$, then up to constants $\phi_0(x) = 1-x$. This function satisfies the Robin condition if and only if $z = -1$.

The final case is $(C(z) + \eta^2) \psi_\eta(x) = 0$ for $\eta > 0$, which corresponds to $(C(z) - \nu^2) \psi_\eta(x) = 0$ when $\nu = \ii \eta$. In this case, again using $\psi_\eta(1) = 0$, we have up to constants
$$
\psi_\eta(x)=\sinh \eta (1-x)\,.
$$
Here $\eta$ is determined as above by the Robin condition at $0$:
$$
\nu  \cosh \eta = - z \sinh \eta\,.
$$
The function $\eta \coth \eta$ is increasing from $1$ to $+\infty$ for $\eta \in [0, +\infty[$, so we have a positive solution if and only if $z < -1$. The corresponding eigenvalue of $C(z)$ is $\nu = \ii \eta$ (up to sign).

In the three cases ($z > -1$, $z = -1$, and $z < -1$), we have determined the smallest eigenvalue $\nu$ of $C(z)$ (which has the same domain as $B(z)$). To obtain the smallest eigenvalue of $B(z)$, we add $z^2$, and to find the norm of the resolvent $\|(A - z)^{-1}\|$, we raise $z^2 + \nu^2$ to the power $-1/2$.

To show that $(z^2 + \nu^2)^{-1/2}$ is decreasing in $z$, it suffices to show that $\nu(z)$ is increasing in $z$. This follows from observing that $\nu \cot \nu$ is decreasing for $\nu \in ]0, \pi[$ (because $(\nu \cot \nu)' = \frac{1}{2}\sin 2\nu - \nu$ and $\frac{1}{2}\sin 2\nu < \nu$ when $\nu > 0$) and that $\eta \operatorname{coth} \eta$ is increasing for $\eta \in ]0, +\infty[$ (for similar reasons).
\end{proof}

\begin{remark}
From \eqref{eq:diffop_r_asymptotic} and \eqref{eq:diffop_r}, we see that
$$
 r(\omega)\sim 2 |\omega| e^{-|\omega|}\,, \quad \omega \to -\infty,
$$
as stated in Theorem 14.3 in \cite{He1}. With Theorem \ref{thm:sg_vanishes}, this confirms that $S(t) = 0$ for every $t \geq 1$.
\end{remark}

 As $z \rightarrow -\infty$ we get $\nu(z) \sim -z$, but the corresponding eigenvalue is $- \nu(z)^2$ and we are at the end interested in $-\nu(z)^2 + z^2$. If we observe that
 $$
 \nu(z) +z \sim - 2 |z| e^{-2 |z|} \mbox{as } z\rightarrow -\infty\,,
 $$
 we get that
 $$
 \nu(z)^2 + z^2 \sim  4  z^2  e^{-2 |z|} \,.
 $$
 and we get
 $$
 r(z)\sim 2 |z| e^{-|z|}\,,
 $$
 as stated in Theorem 14.3 in \cite{He1}. 

\begin{remark}
 On the line, we can consider the family of operators 
\[
	A_n = \frac{\dd}{\dd x} - x^{2n}.
\]
In the limit $n \to \infty$, the function $-x^{2n}$ becomes the negative square well
\[
	V_\infty(x) = \begin{cases} 0, & x \in ]-1, 1[, \\ -\infty, & |x| \geq 1.\end{cases}
\]
One can show that in the limit $n\rightarrow +\infty$ we recover up to a dilation the differentiation model. Explicit computations can be done for this model (see \cite{HV})
\end{remark}

\section{Combining iteration and the semigroup property}\label{s:semigroupization}

\subsection{Semigroup property for $m$} In this subsection we study whether an upper bound $m(t)$ can be improved through the elementary observation that, when $S(t)$ is a one-parameter semigroup, 
\[
	\|S(t_1+t_2)\| \leq \|S(t_1)\| \|S(t_2)\|\,.
\]

Suppose that
\begin{equation}\label{sgpm.1.5}
	m:\, [0,+\infty [\to ]0,+\infty [ \textnormal{ is continuous with } m(0) = 1.
\end{equation}
We assume also that
\begin{equation}\label{sgpm.1}
\| S(t) \|\le m(t),\ \ t\ge 0\,,
\end{equation}
where as usual $S(t)$ is a strongly continuous semigroup of operators with $S(0) = I$.
\subsubsection{Semigroupization}
If $t=t_1+t_2$, $t_j\ge 0$,
we have
$$
\| S(t) \|\le \| S(t_1)\| \| S(t_2)\| \le m(t_1)m(t_2),
$$
so
$$
\| S(t) \|\le \widetilde{m}_2(t),
$$
where
$$
\widetilde{m}_2(t)=\inf_{t_1+t_2=t}m(t_1)m(t_2),
$$
where it is understood that the $t_j$ are restricted to $[0,+\infty
[$. Since $m(0)=1$, we have $\widetilde{m}_2(t)\le m(t).$ Put
$\widetilde{m}_1=m$ and define for $N \ge 2$,
\begin{equation}\label{deftilde}
\widetilde{m}_N(t)=\inf_{t_1+\dots +t_N=t} m(t_1)m(t_2)\cdots m(t_N).
\end{equation}
The continuity of $m$ implies that $\widetilde m_N$ is continuous. 
Clearly, $\widetilde{m}_{N+1}\le \widetilde{m}_N$, and
$$
\| S(t) \|\le \widetilde{m}_N(t).
$$

\par If $N_1,\dots,N_M\in \{ 1,2,\dots \}$, $N=N_1+N_2+\dots+N_M$, then
\begin{equation}\label{sgpm.2}
\widetilde{m}_N(t)=\inf_{s_1+...+s_M=t} \widetilde{m}_{N_1}(s_1)\widetilde{m}_{N_2}(s_2)...\widetilde{m}_{N_M}(s_M).
\end{equation}
From this we get
\begin{proposition}\label{sgpm1}  If for some $N\in \{1,2,\dots \}$ we have $\widetilde{m}_{N+1}=\widetilde{m}_N$, then
$\widetilde{m}_M=\widetilde{m}_N$ for all $M>N$.
\end{proposition}
\medskip
\par Recall that $N\mapsto \widetilde{m}_N$ is decreasing and define
\begin{equation}\label{sgpm.3}
\widetilde{m}_\infty:=\lim_{N\to \infty } \widetilde{m}_N.
\end{equation}
Since the $\widetilde m_N$'s are continuous, $\widetilde m_\infty$ is upper semi-continuous. 
By construction, $\| S(t) \|\le \widetilde{m}_\infty (t)$ if
(\ref{sgpm.1.5}) and (\ref{sgpm.1}) hold. Moreover,
\begin{equation}\label{sgpm.4}
  \widetilde{m}_\infty (t)=\inf_{s_1+...+s_M=t} \widetilde{m}_\infty (s_1)... \widetilde{m}_\infty (s_M),
\end{equation}

From now on we denote by $\mathfrak S$ the map associating with $m$ the function   $\widetilde{m}_\infty$ and note that
\begin{equation}
\mathfrak S \circ \mathfrak S = \mathfrak S\,.
\end{equation}
We say that $m$ is the $\mathfrak S$-invariant if 
$\mathfrak S (m)=m$. 

\begin{remark}
We can modify the definitions above by restricting $s_j$ to
certain finite $t$-dependent sets. For $N=1,2,2^2,2^3,...$, let
$\Gamma _N(t)= \mathbb N \frac{t}{N}$, where $\mathbb N=\{0,1,2,... \}$. For $t\ge
0$, define $\widehat{m}_1(t)=m(t)$,
$$
\widehat{m}_N(t)=\inf_{t_1,...,t_N\in \Gamma _N(t),\atop
  t_1+...+t_N=t} m(t_1)...m(t_N).
$$
Notice that
$$
\widehat{m}_2(t)=\min \left( m(t/2)^2,m(t) \right).
$$
Again, $N\mapsto \widehat{m}_N$ is decreasing, and if (\ref{sgpm.1.5}) and (\ref{sgpm.1}) hold, then $\| S(t) \|\le \widehat{m}_N(t)$.
\par Clearly, $\widehat{m}_N(t)\ge \widetilde{m}_N(t)$, so $$\lim_{N\to
\infty }
\widehat{m}_N(t)\ge \lim_{N\to \infty }\widetilde{m}_N(t) \,.
$$
 On the
other hand, due to the continuity of $m$, we see that for every
$N_0\in \mathbb N\setminus \{0 \}$,  $$ \lim_{N\to
\infty }
\widehat{m}_N(t)\le \widetilde{m}_{N_0}(t) \,,$$ hence
$$
\lim_{N\to
\infty }
\widehat{m}_N(t) = \lim_{N\to \infty }\widetilde{m}_N(t). 
$$
\end{remark}

\subsubsection{Semi-groupization in an interval}
We say that $m$ is $\mathfrak S$-invariant on $[0,T]$, if $m(t_1+t_2)\le m(t_1)m(t_2)$ for
  all $t_1,t_2\in [0,T]$ with $t_1+t_2\in [0,T]$. 
 \par We have
 \begin{equation}\label{dsg.2}
   \widetilde{m}_\infty (t)=\inf_{N\ge 1}
   \inf_{t_1,t_2,...,t_N\ge 0 \atop
t_1+...+t_N=t}\widetilde{m}_\infty(t_1)\widetilde{m}_\infty(t_2)...\widetilde{m}_\infty(t_N).
\end{equation}
In fact, it suffices to use the definition \eqref{deftilde}-(\ref{sgpm.3}) for each
factor $\widetilde{m}_\infty(t_j)$ in the right-hand side of (\ref{dsg.2})\,.

\par More generally, for $0<s\le t\le T$, we may put
\begin{equation}\label{dsg.3}
 \mathfrak S(m,s,t)=\inf_{N\ge 1}
\inf_{s\ge t_1,t_2,...,t_N\ge 0 \atop
t_1+...+t_N=t}m(t_1)m(t_2)...m(t_N).
\end{equation}
As above we check that
\begin{equation}\label{dsg.4}
\mathfrak S(\widetilde{m}_\infty ,s,t)=\mathfrak S(m,s,t).
\end{equation}

\subsubsection{Semi-groupization in a discrete setting} 

We now discuss the question of approximating $\mathfrak{S}(m)$ with a finite number of operations. Let
\[
	[[m, n]] = [m, n] \cap \Bbb{Z}
\]
and suppose that we wish to approximate $\mathfrak{S}(m)$ on the discretized half-line $h\Bbb{N}$. We would naturally define for all $t \in h\Bbb{N}$
\[
	\widetilde{m}_{\infty, h}(t) = \mathop{\inf_{t_1, \dots, t_K \in h\Bbb{N}\backslash\{0\}}}_{t_1 + \dots + t_K = t} m(t_1)m(t_2)\cdots m(t_K).
\]
Whether $\widetilde{m}_{\infty, h}(t) \to \widetilde{m}_{\infty}(t)$ as $h \to 0^+$ would depend of course on the continuity properties of $m$.

A direct approach to evaluating this formula would involve a number of terms equal to the partition function of $t/h$ (which grows exponentially rapidly in $\sqrt{t/h}$), but this is unnecessary. If $t_1, \dots, t_K \in h\Bbb{N}\backslash\{0\}$ with $t_1 + \dots + t_K = t$, we can assume without loss of generality that the $t_j$ are increasing and that therefore either $t_1 = t$ or $t_1 \leq t/2$. In the latter case,
\[
	m(t_1)m(t_2)\cdots m(t_K) \leq \widetilde{m}_{\infty, h}(t_1)\widetilde{m}_{\infty, h}(t-t_1),
\]
and the former case $t_1 = t$ formally corresponds to $t_1 = 0$ and $t_2 = t$. We therefore see that for any $s \in h\Bbb{N} \backslash\{0\}$,
\[
	\widetilde{m}_{\infty, h}(s) = \min\{m(s), \min\{\widetilde{m}_{\infty, h}(s_1)\widetilde{m}_{\infty, h}(s-s_1)\:: s_1 \in h\Bbb{N}, h \leq s_1 \leq s/2\}\}.
\]
For example, using that $m(0) = 1$, we compute 
$$
\begin{array}{ll}
\widetilde{m}_{\infty, h}(h) &= m(h), \\
\widetilde{m}_{\infty, h}(2h) &= \min\{m(2h), \widetilde{m}_{\infty, h}(h)^2\},\\
\widetilde{m}_{\infty, h}(3h) & = \min\{m(3h), \widetilde{m}_{\infty, h}(h)\widetilde{m}_{\infty, h}(2h)\}, \\

	\widetilde{m}_{\infty, h}(4h) &= \min\{m(4h), \widetilde{m}_{\infty, h}(h)\widetilde{m}_{\infty, h}(3h), \widetilde{m}_{\infty, h}(2h)^2\},
\end{array}
$$
and so on.\\
 When $t = Nh$, we therefore need at most $\sum_{k = 0}^{N} k/2 = \frac{1}{4}N(1+N)$ terms to compute $\widetilde{m}_{\infty, h}(t)$.

\subsection{An iteration scheme}

Let $m$ satisfy (\ref{sgpm.1.5}) and (\ref{sgpm.1}) and suppose in addition that $\log m(t)$ is piecewise differentiable. We wish to study iterating applications of Theorem \ref{thm:Riccati} and semigroupization. 

\begin{remark}
In order to apply Theorem \ref{thm:Riccati} to our iteration, we would need to assume that $\mathfrak{S}$ yields a continuous function, or we would need to develop a theory with weaker assumptions on $m$.
\end{remark}

To simplify notation, with $U(m, \omega, r)$ from \eqref{eq:thm_Riccati_U}, let
\[
	U_\omega(m)(t) = U(m, \omega, r)(t) = \begin{cases} 
	m(t), & 0 \leq t \leq 2a^*,
	\\
	\min\{m(t), \ee^{(\omega - r)(t-2a^*)}m(a^*)^2\}, & t > 2a^*.
	\end{cases}
\]
assuming that $r = r(\omega)$ is determined by $\omega$, as in Definition \ref{def:r}. The constant $a^* = a^*(m, \omega, r)$ is as in \eqref{eq:def_astar}. Examining \eqref{eq:def_phi}, it is clear that $a^*$ does not depend on the values of $m(t)$ for $t > a^*$.

If we have a nonempty set $\Omega \subset \Bbb{R}$ of values of $\omega$, we put
$$
\underline{U}_\Omega (m) = \inf_{\omega \in \Omega }U_\omega (m).
$$
\par We shall consider the iteration with a fixed set $\Omega \subset \mathbb R\,$, given by
\begin{equation}\label{iter.0}
m\mapsto (\mathfrak S \underline{U}_\Omega  )^km,\ k=1,2,3,...
\end{equation}
Observe that $\mathfrak S U_\omega $ and $ \mathfrak S\underline{U}_\Omega $ are decreasing
operations,
\begin{equation}\label{iter.1}
 \mathfrak S\underline{U}_\Omega m\le \mathfrak S U_\omega m\le m,\ \forall \omega \in
\Omega .
\end{equation}
Note that
\begin{equation}\label{eq:sginf}
\mathfrak S(\inf_{\nu \in \mathcal{N}}m_\nu )\le \inf_{\nu \in \mathcal{N}}\mathfrak S(m_\nu
), 
\end{equation}
if $\left( m_\nu\right)_{\nu \in \mathcal{N}}$ is a family of functions $[0,+\infty
[\to [0,+\infty  [$.
We define
$$\underline{a}^*_\Omega (m) =\inf_{\omega \in \Omega }a^*(m,\omega ).$$

Assume from now on that $\Omega \subset \mathbb R $ is finite and
non-empty. Let
\begin{equation}\label{iter.2}\underline{\Omega }(m)=\{\omega \in
  \Omega ;\, a^*(m,\omega )=\underline{a}^*_\Omega (m) \} .
\end{equation}
By definition, $U_\omega m = m$ on $[0, 2 a^*(m, \omega)]$. Therefore, if $m$ is $\mathfrak{S}$-invariant, then
\begin{equation}\label{iter.3}
\mathfrak S U_\omega m=m\hbox{ on }[0,2\underline{a}^*_\Omega (m)],\ \forall
\omega \in \Omega,
\end{equation}
so
$$
\mathfrak S\underline{U}_\Omega m=m\hbox{ on }[0,2\underline{a}^*_\Omega (m)].
$$
From (\ref{iter.3}) it follows that
\begin{equation}\label{iter.4}
a^*(\mathfrak S\underline{U}_\Omega m,\omega )\begin{cases}\ge
  \underline{a}^*_\Omega (m),\ \forall \omega \in \Omega ,\\
= \underline{a}^*_\Omega (m),\ \forall \omega \in \underline{\Omega }(m).
\end{cases}
\end{equation}

Let $\omega _0\in \underline{\Omega }(m)$. Then
\begin{equation}\label{iter.5}
  U_{\omega _0}m=\begin{cases}
    m\hbox{ on }[0,2\underline{a}_\Omega ^*(m)],\\
    \min \left( m,m(\underline{a}_\Omega ^*(m))^2 e^{(\omega _0-r(\omega
      _0))(t-2\underline{a}_\Omega ^*(m))}\right)\hbox{ on
    }]2\underline{a}_\Omega ^*(m),+\infty [.
  \end{cases}
\end{equation}
Here we observe that if $\omega \in \Omega $ and $\widetilde{m}:\,
[0,+\infty [\to [0,+\infty [$ satisfies
\begin{equation}\label{iter.6}
  \widetilde{m} \  \begin{cases}
    =m \hbox{ on }[0,2a^*(m,\omega )],\\
    \le U_\omega m\hbox{ on }]2a^*(m,\omega ),+\infty [ ,
  \end{cases}                                                            
\end{equation}
then $a^*(\widetilde{m},\omega )=a^*(m,\omega )$, so $U_\omega
\widetilde{m}$ is well defined with the same value of $a^*
$ as in the definition of $U_\omega m$, and we have
\begin{equation}\label{iter.7}
U_\omega \widetilde{m}=\widetilde{m}.
\end{equation}
If in addition $\mathfrak S\widetilde{m}=\widetilde{m}$, then
\begin{equation}\label{iter.8}
\mathfrak S U_\omega \widetilde{m}=\widetilde{m}.
\end{equation}

\par This can be applied to $\widetilde{m}=\mathfrak S U_{\omega _0}m$ which
satisfies (\ref{iter.6} with $\omega =\omega _0$, hence by
(\ref{iter.8}), we have
\begin{equation}\label{iter.9}
(\mathfrak S U_{\omega _0})^2m=\mathfrak S U_{\omega _0}m.
\end{equation}
A first conclusion is then:
\begin{proposition}\label{iter1}
Let $m:\, [0,+\infty [\to ]0,+\infty [$ satisfy $m(0)=1$ and
$\mathfrak S$-invariant. Let $\Omega =\{ \omega _0 \}$ contain a single
frequency. Then (\ref{iter.9}) holds, so the iteration $m\mapsto
(\mathfrak S U_{\omega _0})^km$, $k=1,2,...$ becomes stationary after the first
step ($k=1$):
$$(\mathfrak S U_{\omega _0})^km=\mathfrak S U_{\omega _0}m,\ \forall k\ge 1.$$
\end{proposition}

\par We now return to the general case with $\Omega $ finite and $m$
as above. Let $\omega _0\in \underline{\Omega }(m)$. Then
$\widetilde{m}:=\mathfrak S\underline{U}_\Omega m$ satisfies (\ref{iter.6}) with $\omega
=\omega _0$ and so does $(\mathfrak S\underline{U}_\Omega )^km$
for all $k\ge 1$. By (\ref{iter.8}), with $\omega =\omega _0$, we
then have
$$\mathfrak S U_{\omega _0}(\mathfrak S\underline{U}_\Omega )^km=(\mathfrak S\underline{U}_\Omega )^km.$$
This holds for all $\omega _0\in \underline{\Omega }(m)$, so if
$\underline{\Omega }(m)=\Omega $, then
$$
(\mathfrak S\underline{U}_\Omega )^{k+1}m=(\mathfrak S\underline{U }_\Omega)^k
m=\mathfrak S\underline{U }_\Omega m,\ \forall k\ge 1
$$
and the iteration procedure becomes stationary after the first step
($k=1$).

If $\underline{\Omega }(m)\ne \Omega $, then we get
\begin{equation}\label{iter.10}
(\mathfrak S\underline{U}_\Omega )^{k+1}m= (\mathfrak S\underline{U}_{\Omega\setminus
  \underline{\Omega }(m)} )^k \mathfrak S\underline{U}_\Omega m,\ \forall k\ge 0.
\end{equation}
This means that after the first step, we can continue the iteration
after replacing $(m,\Omega )$ with $(\widetilde{m},\widetilde{\Omega
})=(\mathfrak S\underline{U}_\Omega m,\Omega \setminus \underline{\Omega }(m))$
It follows that the iteration becomes stationary after at most
$\#\Omega $ steps.
\begin{proposition}\label{iter2}
Let $m:[0,+\infty [\to ]0,+\infty [$ satisfy $m(0)=1$ and 
$\mathfrak S$-invariant. Assume that $\Omega $ is finite. Then the iteration
$m\mapsto (\mathfrak S\underline{U}_\Omega )^km$ is stationary for $k\ge
\#\Omega $, i.e.
$$
(\mathfrak S\underline{U}_\Omega )^km=(\mathfrak S\underline{U}_\Omega )^{\# \Omega
}m,\hbox{ for }k\ge \#\Omega .
$$
\end{proposition}

\subsection{Preservation of log-concave upper bounds}

It is well-known that a concave function $f:[0, +\infty[ \to \Bbb{R}$ satisfying $f(0) \geq 0$ is subadditive, since, by writing $s= \frac{s}{s+t} (s+t) + \frac{t}{s+t} \cdot 0$ and similarly for $t$,
\[
	f(s) + f(t) = f(\frac{s}{s+t}(s+t)) + f(\frac{t}{s+t}(s+t)) \geq f(s+t). 
\]
Therefore whenever $\log m$ is concave, $m$ is $\mathfrak{S}$-invariant. We now show that this property is preserved by $U$, and one does not need to apply $\mathfrak{S}$ if one begins the iterations with a $\log$-concave function.

\begin{proposition}
If $m:[0, +\infty[ \to [0, +\infty[$ is log-concave, then for any $r \geq 0$ and $\omega \in \Bbb{R}$, the function $U(m, \omega, r)(t)$ in \eqref{eq:thm_Riccati_U} is also log-concave.
\end{proposition}

\begin{proof}
If $a^* = +\infty$ then $U(m, \omega, r) = m$; the claim is trivially true in this case. We therefore suppose that $a^* < +\infty$.

Let
\[
	\ell(t) = 2\log m(a^*) + (\omega - r)(t - 2a^*),
\]
so that $\log U(m, \omega, r)(t) = \min\{\log m(t), \ell(t)\}$ for all $t \geq 2a^*$. By the semigroup property, $\log m(2a^*) \leq \ell(2a^*)$. By concavity of $\log m$ there exists some $\lambda \in \Bbb{R}$, for instance the derivative from the right of $\log m(t)$ at $2a^*$, such that 
\[
	\log m(t) \leq \log m(2a^*) + \lambda (t-2a^*), \quad \forall t \geq 0.
\]

If $\lambda \leq \omega - r$, then $\log m(t) \leq \ell(t)$ for all $t \geq 2a^*$. This implies that $U(m, \omega, r) = m$, so $U$ is again automatically log-concave. If $\lambda \geq \omega-r$, then for all $t \in [0, 2a^*]$, $\log m(t) \leq \ell(t)$. Therefore 
\[
	\log U(m, \omega, r)(t) = \min\{\log m(t), \ell(t)\} \mbox{ for all } t \geq 0\,.
\]
This implies that $\log U$ is concave since it is the minimum of two concave functions, which completes the proof.
\end{proof}

\section{Iterating Theorem \ref{thm:Riccati} when $\log m(t)$ is piecewise affine}\label{s:paffine}

\newcommand{\update}{U}
\newcommand{\mRiccati}{M}

In this section, we apply Theorem \ref{thm:Riccati} iteratively to upper bounds $m(t)$ such that $\log m(t)$ is piecewise affine and concave. The starting point is applying Theorem \ref{thm:Riccati} with $r(0) \leq 1$ to $m(t) = 1$, which gives the upper bound of D.~Wei in Theorem \ref{th3.2}.

\subsection{Solving the Riccati equation in an interval when $\mu$ is constant}\label{ss:mucst_general}

To begin, we solve equation \eqref{defPhi} for constant $\mu$. We will then consider the general case by translation and dilation. We begin by considering the regions in the $(\mu, \Phi)$-plane given by considering the autonomous differential equation
\[
	\begin{cases}
	\Phi'(b) = (\Phi(b)^2 + 2\mu \Phi(b) + 1),
	\\
	\Phi(0) = \Phi_0
	\end{cases}
\]
as in \eqref{defPhi}. In Proposition \ref{prop:phi_mucst} we record the explicit forms of solutions to this differential equation according to these regions with initial data $\Phi_0$.

In Figure \ref{fig:mu_Phi}, we plot the direction of $\Phi'$ as a function of $\mu$ and $\Phi$. One sees immediately that a solution with $\Phi_0 \geq 0$ will remain positive, and that the solution with $\Phi_0 = 0$ will arrive at $\Phi = 1$ in finite time if and only if $\mu > -1$.

\begin{figure}
\centering
\includegraphics[width = .6\textwidth]{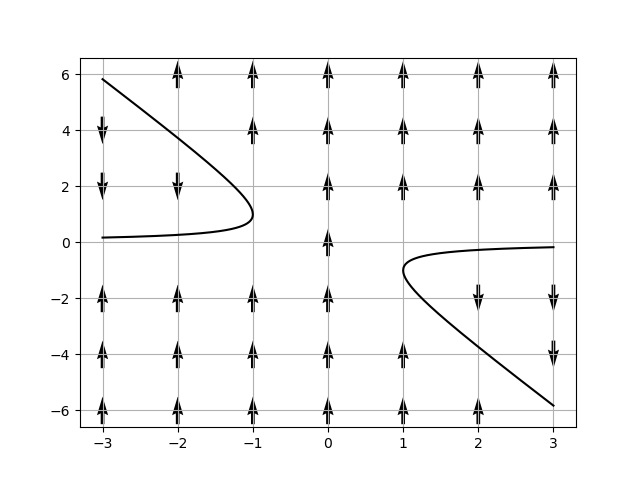}
\caption{\label{fig:mu_Phi} Plot of direction of $\Phi' = \Phi^2 + 2\mu \Phi + 1$ in the $(\mu, \Phi)$ plane.}
\end{figure}

\begin{proposition}\label{prop:phi_mucst}
For $\mu, \Phi_0 \in \Bbb{R}$, let
\[
	\Phi(b; \mu, \Phi_0), \quad b \in ]0, +\infty[
\]
be the solution to
\begin{equation}\label{eq:phi_mucst}
	\begin{cases}
	\Phi'(b) = \Phi(b)^2 + 2\mu \Phi(b) + 1, & b > 0,
	\\
	\Phi(0) = 0.
	\end{cases}
\end{equation}
\begin{enumerate}[(a)]
\item \label{it:phi_mucst_mul1} If $\mu^2 < 1$ then, with $\eta = \sqrt{1-\mu^2}$,
\[
	\Phi(b) = \eta \tan(\eta b + c)- \mu, \quad c = \operatorname{arctan} \frac{\Phi_0 + \mu}{\eta}.
\]
\item \label{it:phi_mucst_mue1} If $\mu^2 = 1$ then $\Phi(b)$ is constant if $\Phi_0 = -\mu$ and otherwise
\[
	\Phi(b) = -\frac{1}{b+c} - \mu, \quad c = -\frac{1}{\Phi_0 + \mu}.
\]
\item \label{it:psi_mucst_mug1} If $\mu^2 > 1$ then, with $\eta = \sqrt{\mu^2 - 1}$, then
\begin{enumerate}[(i)]
\item \label{it:psi_mucst_mug1_cst} if $|\Phi_0 + \mu| = \eta$, then $\Phi \equiv \Phi_0$ is constant;
\item \label{it:psi_mucst_mug1_coth} if $|\Phi_0 + \mu| > \eta$, then
\[
	\Phi(b) = \eta \coth(-\eta b + c) - \mu, \quad c = \operatorname{arccoth} \frac{\Phi_0 + \mu}{\eta};
\]
and
\item \label{it:psi_mucst_mug1_tanh} if $|\Phi_0 + \mu| < \eta$, then
\[
	\Phi(b) = \eta \tanh(-\eta b + c) - \mu, \quad c = \operatorname{arctanh}\frac{\Phi_0 + \mu}{\eta}.
\]
\end{enumerate}
\end{enumerate}
\end{proposition}

In order to apply Theorem \ref{thm:Riccati} to a pair $(\omega, r)$ and an upper bound $m(t)$ such that $\log m(t)$ is continuous and piecewise affine, we search for $b^* = b^*(\mu; \omega, r) \in ]0, +\infty]$ maximal such that $\Phi(b) < 1$ on $[0, b^*[$. To this end, we solve \eqref{defPhi} on the intervals on which $\log m(t)$ is affine. It is in principle possible that $\Phi$ tends to $+\infty$ on an interval where $|\mu| < 1$, but in this case $\Phi$ crosses $\Phi = 1$, so $\Phi$ will always be well-defined on $[0, b^*[$.

Suppose that with $0 = t_0 < t_1 < \dots < t_{N-1} < t_N = \infty$ and that 
\[
	\log m(t) = \alpha_j t + \beta_j, \quad t \in ]t_j, t_{j+1}[.
\]
Since we are assuming that $\log m(t)$ is concave, $\alpha_{j+1} < \alpha_j$ for each $j$.
Let
\begin{equation}\label{eq:muj_paffine}
	\mu_j = \frac{1}{r}(\alpha_j - \omega).
\end{equation}
Then on successive intervals $]t_j, t_{j+1}[$ for $j = 0, 1, 2, \dots$, we can compute the solution to \eqref{eq:def_phi} using Proposition \ref{prop:phi_mucst}:
\begin{equation}\label{eq:phi_mucst_scaling}
	\phi(t) = \Phi(r(t-t_j); \mu_j, \phi(t_j)), \quad t \in [t_j, t_{j+1}].
\end{equation}
To begin with $t_0 = 0$, we recall that $\phi(0) = 0$. As one can see from Figure \ref{fig:mu_Phi}, $\phi(t_j) \geq 0$ implies that $\phi(t_{j+1}) > 0$, so $\phi(t) \in ]0, 1[$ for all $t \in ]0, a^*[$.

The only case where $\Phi_0 < 1$ and $\Phi(t) = 1$ for some $t$ positive is when $\mu > -1$ (see Figure \ref{fig:mu_Phi}). On each interval $]t_j, t_{j+1}[$ where $\mu_j > -1$ and $\phi(t_j) < 1$, we can find a candidate $a_j^*$ for the first solution to $\phi(a) = 1$. 
\begin{enumerate}[(a)]
\item If $\mu_j \in ]-1, 1[$, writing $\eta_j = \sqrt{1 - \mu_j^2}$, we have the candidate
\begin{equation}\label{eq:ajstar1}
	a_j^* = t_j + \frac{1}{r\eta_j}\left(\operatorname{arctan} \frac{1+\mu_j}{\eta_j} - \operatorname{arctan}\frac{\phi(t_j) + \mu_j}{\eta_j}\right).
\end{equation}
\item If $\mu_j = 1$, we have the candidate
\begin{equation}\label{eq:ajstar2}
	a_j^* = t_j + \frac{1}{r}\left(\frac{1}{\phi(t_j) + 1} - \frac{1}{2}\right).
\end{equation}
\item If $\mu_j > 1$, writing $\eta_j = \sqrt{\mu_j^2 - 1}$ and recalling that $\phi \geq 0$, we have the candidate
\begin{equation}\label{eq:ajstar3}
	a_j^* = t_j + \frac{1}{r\eta_j}\left(\operatorname{arccoth} \frac{\phi(t_j) +\mu_j}{\eta_j} - \operatorname{arccoth}\frac{1 + \mu_j}{\eta_j}\right).
\end{equation}
\end{enumerate}

The hypothesis that $\log m(t)$ is concave is equivalent to supposing that $\{\mu_j\}_{j = 0}^{N-1}$ is a decreasing sequence. By Proposition \ref{prop:astardecr}, since we are considering $r$ fixed and since changing $\mu_j$ is the same as changing $\omega$ in \eqref{eq:def_mu},
\[
	a_{j+1}^* \geq a_j^*.
\]
We obtain that the first $t > 0$ such that $\phi(t) = 1$ is
\[
	a^* = \min\{a_j^* \::\: a_j^* \in ]t_j, t_{j+1}]\},
\]
with the convention that $a^* = +\infty$ if there is no such $a_j^*$. Note that, if for some $j$ we have $a_j^* > t_{j+1}$ and $\mu_{j+1} \leq -1$, then $a^*_{k} = +\infty$ for every $k > j$ which implies $a^* = +\infty$. 

\subsection{Application of Theorem \ref{thm:Riccati} when $m_0 \equiv 1$}\label{ss:m0is1}

We now perform some explicit computations for the constant function
\[
	m_0(t) = 1.
\]
This is the classical upper bound for $\|\exp(-tA)\|$ when $A$ is $m$-accretive. Of course, $\log m_0$ is piecewise affine and concave, so if we iteratively apply Theorem \ref{thm:Riccati} with varying values of $\omega$ and $r$, we obtain a sequence of upper bounds whose logarithms are piecewise affine and concave.

Given $\omega$ and $r = r(\omega)$, we have
\[
	a^*(m_0, \omega, r) = \frac{1}{r}a_0^*(\mu_0), \quad \mu_0 = -\frac{\omega}{r}.
\]
Then $a^* = a_0^*$ exists if and only if $\mu_0 > -1$, which is equivalent to $\omega - r < 0$. (This condition is natural because otherwise $\ee^{(\omega - r)(t - 2a^*)}m_0(a^*)^2$ could never be better than $m_0(t) = 1$ when $t > 2a^*$.)

As a reference case, if $\omega = 0$, then $\mu_0 = 0$ and by \eqref{eq:ajstar1},
\[
	a^* = \frac{\pi}{4r}.
\]
In this case,
\[
	U(m_0, 0, r)(t) = 
	\begin{cases} 1, & t \in ]0, \frac{\pi}{2r}[, 
	\\ \ee^{-rt + \frac{\pi}{2}}, & t \in ]\frac{\pi}{2r}, +\infty[,
	\end{cases}
\]
which is precisely Theorem \ref{th3.2}.

Using the case $\omega = 0, r = 1$ as a reference, we can examine how $a^*$ and $U(m_0, \omega, r)$ vary as $\omega$ and $r$ vary. 

To begin, if $0 < r \leq \omega$ then $a^* = +\infty$. This condition corresponds to the situation where $\ee^{(\omega - r)(t - 2a)} \geq 1$ whenever $t \geq 2a$, which means that Theorem \ref{thm:Riccati} could never give an improvement over $m_{\rm{old}} \equiv 1$ when $r \leq \omega$.

In the sector $\{(\omega, r) \::\: \omega > r > 0\},$ for $\omega$ fixed, $a^* \to +\infty$ as $r \to \max\{\omega, 0\}$ and $a^* \to 0$ as $r \to \infty$. Since $a^*(m_0, \omega, r)$ is decreasing in $r$ for $\omega$ fixed, for every $\alpha > 0$ we may define $r^* = r^*(\alpha, \omega)$ as the unique $r$ such that
\begin{equation}\label{eq:def_rstar}
	a^*(m_0, \omega, r^*(\alpha, \omega)) = \alpha.
\end{equation}
In Proposition \ref{prop:rstar_explicit} below we show that $r^*(\alpha, \omega)$ may be determined in terms of $r(\omega)$ for the derivative operator in \eqref{eq:3.1}.

In Figure \ref{fig:omegar} we show the boundary of the sector $\{(\omega, r) \::\: \omega > r > 0\}$, the solid curve $\{r = r^*(\frac{\pi}{4}, \omega)\}$, and the line $\{r = \omega + 1, r \geq 0\}$. For use in the next subsection, we include as dotted curves the graphs $\{r = r^*(\frac{\pi}{2}, \omega)\}$ and $\{r = r^*(\frac{\pi}{8}, \omega)\}$, which lie below and above the solid curve $\{r = r^*(\frac{\pi}{4}, \omega)\}$. 

\begin{figure}
\centering
\includegraphics[width = .6\textwidth]{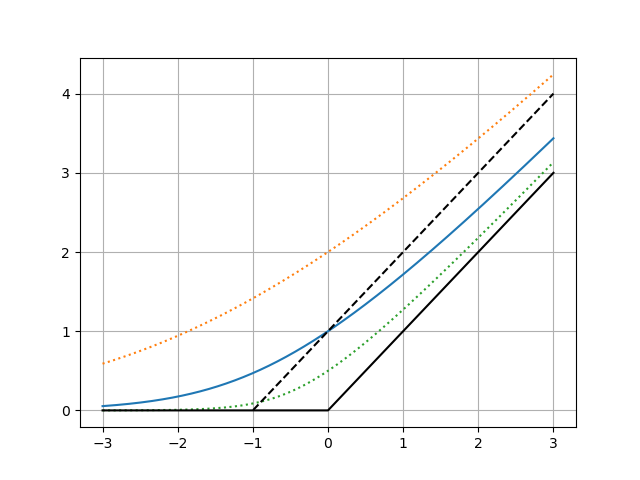}
\caption{\label{fig:omegar} Regions of $(\omega, r)$ where we can compare applications of Theorem \ref{thm:Riccati} with $m_0 \equiv 1$ to $\omega = 0$, $r = 1$.}
\end{figure}

Above the solid curve $\{(\omega, r^*(\frac{\pi}{4}, \omega)\}$, $a^* < \frac{\pi}{4}$, so the estimate $\ee^{(\omega - r)(t-2a^*)}$ for $t \geq 2a^*$ takes effect sooner than if $\omega = 0$ and $r = 1$ (and below the solid curve the estimate takes effect later). To the left of the line $\{r = \omega + 1\}$ the estimate $\ee^{(\omega - r)(t-2a^*)}$ decreases more rapidly than if $\omega = 0$ and $r = 1$, while to the right the estimate decreases more slowly.

We therefore have four regions that we can compare with the case $\omega = 0, r = 1$: below the solid curve $\{r = r^*(\frac{\pi}{4}, \omega)\}$ and to the right of the dashed line $\{r = \omega+1\}$, the estimate $\ee^{(\omega - r)(t-2a^*)}, t \geq 2a^*$ takes effect later and decreases more slowly than the reference case $\ee^{-(t-\frac{\pi}{2})}, t \geq \frac{\pi}{2}$. The estimate from Theorem \ref{thm:Riccati} with $(\omega, r)$ in this region is therefore everywhere larger than the estimate with $\omega = 0, r = 1$. Similarly, above the solid curve and to the left of the dashed line, the estimate from Theorem \ref{thm:Riccati} is everywhere smaller than the reference case. 

It is in the regions above the solid curve and to the right of the dashed line or below the solid curve and to the left of the dashed line, that is, $\{a^*(\omega, r) < \frac{\pi}{4}, r < \omega + 1\}$ and $\{a^*(\omega, r) > \frac{\pi}{4}, r > \omega + 1\}$ that combining Theorem \ref{thm:Riccati} with $(\omega, r)$ and with $\omega = 0, r = 1$ could give new information. We analyze this question in the next subsection.

We also have the following information on the resolvent of a hypothetical semigroup whose norms optimize the estimate in Theorem \ref{th3.2}.

\begin{proposition}
Suppose that $-A$ is a $m$-accretive operator on acting on a Hilbert space. If
\[
	\|\exp tA\| = 
	\begin{cases}
	1, & t \in [0, \frac{\pi}{2}],
	\\
	\exp(\frac{\pi}{2} - t), & t \in ]\frac{\pi}{2}, +\infty[,
	\end{cases}
\]
then
\[
	r(\omega) \leq \begin{cases} r^*(\frac{\pi}{4}, \omega), & \omega > 0,
	\\
	\omega + 1, & -1 < \omega \leq 0.
	\end{cases}
\]
\end{proposition}

\subsection{Iterating from $m_0 \equiv 1$ with two $(\omega, r)$ pairs}

As a reference we take $m_0 \equiv 1$ and $(\omega_1, r_1) = (0, 1)$. As noted previously,
\[
	U(m_0, \omega_1, r_1) = m_1,
\]
where
\[
	\log m_1(t) = 
	\begin{cases} 
	0, & 0 \leq t \leq \frac{\pi}{2},
	\\
	\frac{\pi}{2} - t, & t > \frac{\pi}{2}.
	\end{cases}
\]
The question we consider here is whether iterating estimates from Theorem \ref{thm:Riccati} with $(\omega_1, r_1) = (0, 1)$ and some $(\omega_2, r_2)$ improves the estimates beyond simply taking the minimum  of the estimates obtained separately. Using the notation in \eqref{eq:thm_Riccati_U}, we would like to compare
\[
	\min\{\update(m_0, \omega_1, r_1), \update(m_0, \omega_2, r_2)\}
\]
with
\[
	\update(\update(m_0, \omega_1, r_1), \omega_2, r_2) = \update(m_1, \omega_2, r_2)
\]
or
\[
	\update(\update(m_0, \omega_2, r_2), \omega_1, r_1).
\]
Our answer is that there are situations where there is some improvement, but for relatively large $t$ and for small sets of pairs $(\omega_2, r_2)$.

If $r > \omega + 1$ and if $r > r^*(\frac{\pi}{4}, \omega)$, then $a^*(m_0, \omega, r) < \frac{\pi}{4}$ and $\omega - r < -1$. Therefore $a^*(m_1, \omega, r) = a^*(m_0, \omega, r)$ and $\update(m_1, \omega, r) = \update(m_0, \omega, r)$ is less than $m_1$ on $t \geq 2a^*(m_0, \omega, r)$. Consequently,
\[
	\update(m_0, \omega, r) = \update(m_1, \omega, r) = \min\{m_1, \update(m_0, \omega, r)\}.
\]
It is therefore $\update(m_0, \omega, r)$ which is everywhere stronger as an upper bound than $m_1$. For similar reasons, if $r < \omega + 1$ and if $r > r^*(\frac{\pi}{4}, \omega)$, then
\[
	m_1 = \update(\update(m_0, \omega, r), 0, 1) = \min\{m_1, \update(m_0, \omega, r)\},
\]
and it is $m_1$ which provides the smaller upper bound.

We next consider the case of those $(\omega_2, r_2)$ which satisfy
\[
	\max\{r^*(\frac{\pi}{2}, \omega_2), 1+\omega_2\} \leq r_2 \leq r^*(\frac{\pi}{4}, \omega_2).
\]
Note that this implies that $\omega \leq 0$. In this case $m_1 = \update(m_0, 0, 1)$ and $\update(m_1, 0, 1)$ are both constant and equal to $1$ on $[0, \frac{\pi}{4}]$. Consequently, the values of $a^*$ are unchanged (that is, $a^*(m_0, 1, 0) = a^*(\update(m_0, \omega_2, r_2), 1, 0)$ and $a^*(m_0, \omega_2, r_2) = a^*(m_1, \omega_2, r_2)$). Therefore
\begin{equation}\label{eq:logm_paffine_nonewinfo}
	\update(m_1, \omega_2, r_2) = \update(\update(m_0, \omega_2, r_2), 0, 1) = \min\{m_1, \update(m_0, \omega_2, r_2)\}.
\end{equation}
In other words, taking the minimum of the estimates obtained from $(\omega_1, r_1) = (0, 1)$ and $(\omega_2, r_2)$ gives a different and superior upper bound to the estimates taken separately, but iterating the procedure coming from Theorem \ref{thm:Riccati} gives no new information.

For the same reasons, if
\[
	r^*(\frac{\pi}{4}, \omega_2) \leq r_2 \leq \min\{\omega_2 + 1, r^*(\frac{\pi}{8}, \omega_2)\},
\]
(which implies $\omega \geq 1$) then \eqref{eq:logm_paffine_nonewinfo} holds as well.

The only cases in which iterating Theorem \ref{thm:Riccati} could provide a better estimate are when
\[
	\max\{0, \omega_2 + 1\} < r_2 < r^*(\frac{\pi}{2}, \omega_2)
\]
or when
\[
	r^*(\frac{\pi}{8}, \omega_2) < r_2 < \omega_2 + 1.
\]
This is a very restrictive set of $(\omega_2, r_2)$: for instance, the former requires that $\omega < -0.8891$ and the latter requires that $\omega > 4.7391$. Numerically it seems that, subject to these restrictions, one always has an improvement from iterating Theorem \ref{thm:Riccati}. However, this improvement is modest and applies to large $t$. We illustrate this phenomenon with an example.

\begin{example}
Let $(\omega_2, r_2) = (-1, 0.05)$. With $\mu = -\omega_2/r_2 = 20$ and $\eta = \sqrt{\mu^2 - 1} = \sqrt{399}$ we obtain
\begin{equation}\label{ex:twoitersa2}
\begin{aligned}
	a^* &= a^*(m_0, \omega_2, r_2) = \frac{1}{r_2\eta}\operatorname{arccoth}\frac{1 + \mu}{\eta} 
	\\ &= \frac{20}{\sqrt{399}}\left(\operatorname{arccoth}\frac{20}{\sqrt{399}} - \operatorname{arccoth}\sqrt{\frac{21}{19}}\right) 
	\\ &\approx 1.8464.
	\end{aligned}
\end{equation}
(The pair $(\omega_2, r_2)$ was chosen so that $a^*_2 > \frac{\pi}{2}$.) Therefore $\update(m_0, \omega_2, r_2) = m_2$ where
\[
	\log m_2(t) = 
	\begin{cases}
	0, & 0 \leq t \leq 2a^*,
	\\
	-1.05(t - 2a^*), & t > 2a^*.
	\end{cases}
\]
If we attempt to update $m_2$ using the data $(\omega_1, r_1) = (0, 1)$, we obtain $a^*(m_2, 0, 1) = a^*(m_0, 0, 1) = \frac{\pi}{4}$ (because $m_2 = 1$ on $[0, \frac{\pi}{4}]$). Therefore
\[
	m_{2, 1} = U(\omega_1, r_1, m_2)
\]
is simply
\[
	m_{2, 1}(t) = \min\{m_1(t), m_2(t)\}
	= \begin{cases}
	1, & 0 \leq t \leq \frac{\pi}{2},
	\\
	\ee^{\frac{\pi}{2} - t}, & \frac{\pi}{2} < t \leq t_2
	\\
	\ee^{-1.05(t-2a^*)}, & t > t_2.
	\end{cases}
\]
Here $t_2 \approx 46.1344$ is the point where $\frac{\pi}{2} - t = -1.05(t-2a^*_2)$.

In the reverse order, we can compute $\update(m_1, \omega_2, r_2)$. We begin by computing $\phi$ satisfying the Riccati equation 
\[
	\phi' = -r_2(\phi^2 + 2\mu \phi + 1),
\]
where, following \eqref{eq:muj_paffine},
\[
	\mu(t) = 
	\begin{cases}
	-\frac{\omega_2}{r_2} = 20, & 0 \leq t \leq \frac{\pi}{2},
	\\
	\frac{1}{r_2}(-1 - \omega_2) = 0, & t > \frac{\pi}{2}.
	\end{cases}
\]
In this case, in the notation of Proposition \ref{prop:phi_mucst},
\[
	\begin{aligned}
	\phi(\frac{\pi}{2}) 
	&= 
	\Phi(\frac{\pi}{2}r_2 ; 20, 0) 
	\\ &= 
	\sqrt{399}\operatorname{coth}\left(-\frac{\sqrt{399}}{20}\frac{\pi}{2} + \operatorname{arccoth} \frac{20}{\sqrt{399}}\right) - 20 
	\\ &\approx 
	0.5597.
	\end{aligned}
\]
Continuing on the interval $[t_1, t_2[ = [\frac{\pi}{2}, +\infty[$ where $\alpha_1 = -1$ and $\beta_1 = \frac{\pi}{2}$, we follow \eqref{eq:ajstar1} with 
\[
	\mu_1 = \frac{1}{r_2}(\alpha_1 - \omega_2) = \frac{1}{r}(-1 - (-1)) = 0
\]
and $\eta_1 = 1$. We obtain
\[
	a_1^* = a^*(m_1, \omega_2, r_2) = \frac{\pi}{2} + \frac{1}{r_2}\left(\arctan 1  - \operatorname{arctan}\phi(\frac{\pi}{2})\right) \approx 7.0741.
\]
This is a much larger than the previous value of $a^* \approx 1.8464$ obtained in \eqref{ex:twoitersa2}, but this is compensated for by the fact that $\log m_1(a_1^*) = \frac{\pi}{2} - \tilde{a}^* < 0$. Theorem \ref{thm:Riccati} gives the upper bound $\|S(t)\| \leq \exp M(t)$ for $t \geq 2a_1^*$ with
\[
	\begin{aligned}
	M(t) 
	&= 
	(\omega_2 - r_2)(t - 2a_1^*) + 2(\frac{\pi}{2} - a_1^*)
	\\ & \approx
	-1.05t + 3.8490
	\end{aligned}
\]
Because $3.8490 < 2(1.05)a^* \approx 3.8775$, this represents an improvement over $m_{2, 1} = \min\{m_1, m_2\}$. This improvement is only seen for very large $t$, since
\[
	(\omega_2 - r_2)(t - 2a_1^*) + 2(\frac{\pi}{2} - a_1^*) = \frac{\pi}{2} - t
\]
when $t \approx 45.5641$. 
\end{example}

\subsection{Two examples}\label{ss:paffine_examples}

\subsubsection{A Jordan block}

The algorithm described in Subsection \ref{ss:mucst_general} is sufficiently detailed to be implemented using a computer.

In Figure \ref{fig:Jordan_3} we compare upper bounds for $\|\exp(tJ)\|$ with 
\[
	J = \begin{pmatrix} 0 & 1 & 0 \\ 0 & 0 & 1 \\ 0 & 0 & 0 \end{pmatrix}
\]
a Jordan block of size $n = 3$ and $t \in [0, 20]$. (The norm is the operator norm induced by the usual Euclidean norm on $\Bbb{R}^3$.) The bottom curve is the true value of $\log \|\exp(tJ)\|$ with the norm computed as the largest singular value. 

\begin{enumerate}
\item The highest curve is the numerical range estimate
\[
	t \sup\{\jvRe z\::\: z \in \operatorname{Num}J\} = t \cos\frac{\pi}{n+1} = \frac{1}{\sqrt{2}}t
\]
which is an upper bound for $\log \|\exp(tJ)\|$. 
\item The second-highest curve is the logarithm of the upper bound coming from applying Theorem \ref{thm:Riccati} to the numerical range estimate three times, for $\omega \in \{\frac{1}{2}, 1, 2\}$. 
\item The second-lowest curve is the logarithm of the upper bound from applying the same theorem for the $101$ values
\[
	\omega = \exp(-5.0), \exp(-4.9), \exp(-4.8), \dots, \exp(5.0).
\]
\end{enumerate}

\begin{figure}
\centering
\includegraphics[width = .8\textwidth]{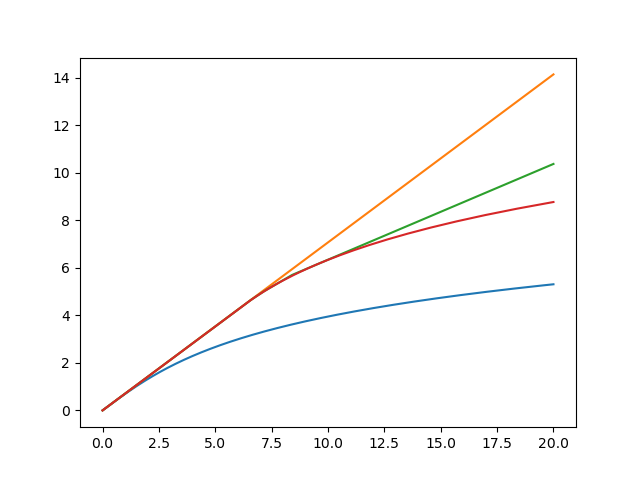}
\caption{\label{fig:Jordan_3} Graph of semigroup norms and upper bounds from Theorem \ref{thm:Riccati} for a Jordan block.}
\end{figure}

At each stage, the upper bounds for $\|\exp(tJ)\|$ improve, but it seems clear that the upper bound after an arbitrary number of applications of Theorem~\ref{thm:Riccati} will remain far from the true value of $\|\exp(tJ)\|$. It also seems to be the case that iterating Theorem \ref{thm:Riccati} does not improve the upper bounds. (We do not attempt to perform an exhaustive analysis of the exact values of $a^*$ and upper bounds in question.)

\subsubsection{The differentiation operator}

As a second example, we consider $A$ the differentiation operator on an interval from \eqref{eq:3.1}. We see that for this generator (and starting from the constant upper bound $m_0 = 1$ for the semigroup) the value of $a^*$ is constant making the estimate in \eqref{eq:thm_Riccati_U} optimal.

\begin{proposition}\label{prop:astar_diffop}
Let $r(\omega)$ be as in Proposition \ref{prop:diffop_svd}, and let $m_0 \equiv 1$ be the constant function. Then
\[
	a^*(m_0, \omega, r) = \frac{1}{2}, \quad \forall \omega \in \Bbb{R}.
\]
\end{proposition}

\begin{proof}
We compute $a_0^*$ from \eqref{eq:ajstar1}--\eqref{eq:ajstar3} with $\mu_0 = -\omega/r(\omega)$, $t_0 = 0$ and $\phi(t_0) = 0$.  If $\omega = -1$ then $r(\omega) = 1$ and $a_0^* = \frac{1}{2}$ by \eqref{eq:ajstar2}. If $\omega > -1$ then $\nu = \nu(\omega) \in ]0, \pi[$ satisfies $-\nu \cot \nu = \omega$ and $\sin \nu > 0$. Consequently
\[
	r(\omega) = \sqrt{\nu^2 \cot \nu^2 + \nu^2} = \frac{\nu}{\sin \nu},
\]
and
\[
	\mu_0 = -\frac{\omega}{r(\omega)} = \frac{\nu \cot \nu}{\nu / \sin \nu} = \cos \nu.
\]
In \eqref{eq:ajstar1}, $\eta_0 = \sqrt{1-\mu_0^2} = \sin \nu$ and
\[
	a_0^* 
	= 
	\frac{1}{(\nu/\sin \nu)\sin \nu}\left(\arctan\frac{1 + \cos \nu}{\sin \nu} - \arctan\frac{\cos \nu}{\sin \nu}\right)
	=
	\frac{1}{2}.
\]
We omit the computation when $\omega < -1$ which is essentially the same.
\end{proof}

We note that, with $r(\omega)$ from Proposition \ref{prop:diffop_svd}, $\omega - r(\omega) < 0$ for all $\omega \in \Bbb{R}$. This follows from computing
\[
	\omega - r(\omega) = -\nu \cot \nu - \frac{\nu}{\sin \nu} = -\nu\frac{\cos\nu + 1}{\sin \nu} = -\nu \cot \frac{\nu}{2}
\]
for $\nu = \nu(\omega) \in \ii ]0, \infty[ \: \cup \: [0, \pi[$ extended by continuity at $\nu(-1) = 0$ to be $-1 - r(-1) = -2$. If $\nu \in ]0, \pi[$ then clearly $-\nu\cot(\nu/2) < 0$, and if $\nu \in \ii]0, \infty[$ then
\[
	-\nu \cot \frac{\nu}{2} = -(\nu/\ii) \coth(\nu/\ii) < 0
\]
as well.

Recall that the semigroup generated by $A$ satisfies $\|S(t)\| = 1$ for $t \in [0, 1[$ and $\|S(t)\| = 0$ for $t > 1$. Therefore the estimate in \eqref{eq:thm_Riccati_U} is optimal in the sense that
\[
	m_0(a^*)\ee^{(\omega - r(\omega))(t - 2a^*)} = \ee^{(\omega - r(\omega)(t - 1)}
\]
is the smallest function of the form
\[
	f_M(t) = M\ee^{(\omega - r(\omega))t}
\]
such that $\|S(t)\| \leq f_M(t)$ for all $t \geq 0$.

We also observe that, in the limit $\omega \to -\infty$, Theorem \ref{thm:Riccati} gives the exact value of $\|S(t)\|$ almost everywhere. Indeed, using \eqref{eq:thm_Riccati_U} and knowing that $\omega - r(\omega) \to -\infty$ as $\omega \to -\infty$,
\[
	\begin{aligned}
	\lim_{\omega \to -\infty} U(m_0, \omega, r(\omega))(t) 
	&= 
	\lim_{\omega \to -\infty} \begin{cases} 1, & 0 \leq t \leq 1,
	\\
	\ee^{(\omega - r(\omega))(t-1)}, & t > 1
	\end{cases}
	\\ &= 
	\begin{cases} 1, & 0 \leq t \leq 1,
	\\ 0, & t > 1.
	\end{cases}
	\end{aligned}
\]

This optimality is preserved under scaling. Continuing to let $A$ denote the operator in \eqref{eq:3.1} and with $\gamma > 0$ and $\delta \in \Bbb{R}$, let
\[
	r(\omega, \gamma A + \delta) = \left(\sup_{\jvRe z > \omega} \|(z-(\gamma A + \delta))^{-1}\|\right)^{-1}
\]
as in Definition \ref{def:r}. It is straightforward to see that, when
\[
	\omega' = \frac{\omega-\delta}{\gamma},
\]
\[
	r(\omega, \gamma A + \delta) = \gamma r(\omega', A).
\]
Moreover, the numerical range of $\gamma A + \delta$ is contained in $\{\jvRe z \leq \delta\}$ because $-A$ is $m$-accretive, so 
\[
	\exp(t(\gamma A + \delta)) \leq m_\delta(t) := \ee^{\delta t}, \quad \forall t \geq 0.
\]
By scaling the semigroup generated by $A$, we have
\[
	\|\exp(t(\gamma A + \delta))\| =
	\begin{cases}
	\ee^{\delta t}, & 0 \leq t < \frac{1}{\alpha},
	\\
	0 & t \geq \frac{1}{\alpha}.
	\end{cases}
\]

Applying Theorem \ref{thm:Riccati} to $\gamma A + \delta$ with upper bound $m_\delta$ then leads us to compute
\[
	a^*(m_\delta, \omega, r(\omega, \gamma A + \delta)) 
	= 
	a^*(m_\delta, \omega, \gamma r(\omega', A)).
\]
The parameter $\mu_0$ given by \eqref{eq:muj_paffine}, using that $\log m_\delta = \delta t$, is
\[
	\frac{1}{\gamma r(\omega', A)}(\delta - \omega) = - \frac{\omega'}{r(\omega', A)}
\]
In Proposition \ref{prop:astar_diffop} (cf.\ \eqref{eq:phi_mucst_scaling} with $j = 0$), we have seen that the first zero of
\[
	t \mapsto \Phi(r(\omega', A) t, \mu_0, 0) - 1
\]
is at $t = \frac{1}{2}$. Therefore the first zero of
\[
	t \mapsto \Phi(r(\omega, \gamma A + \delta)t, \mu_0, 0) - 1
\]
is at
\[
	t = \frac{1}{2}\left(\frac{r(\omega', A)}{r(\omega, \gamma A + \delta)}\right) = \frac{1}{2\gamma}.
\]

In conclusion, for every $\omega \in \Bbb{R}$,
\[
	a^*(m_\delta, \omega, r(\omega, \gamma A + \delta)) = \frac{1}{2\gamma}.
\]
The time $t = 2a^* = \frac{1}{\gamma}$ coincides with the $t$ beyond which $\exp(t(\gamma A + \delta)) = 0$, just as Proposition \ref{prop:astar_diffop} shows in the case $\gamma = 1$ and $\delta = 0$.

Taking the case $\delta = 0$ and $\gamma = \frac{1}{2\alpha}$, one readily obtains the following formula for $r^*(\alpha, \omega)$ from \eqref{eq:def_rstar}.

\begin{proposition}\label{prop:rstar_explicit}
Let $r(\omega)$ be as in Proposition \ref{prop:diffop_svd} and, for $\alpha > 0$ let $r^*(\alpha, \omega)$ solve
\[
	a^*(m_0, \omega, r^*(\alpha, \omega)) = \alpha
\]
with $m_0 \equiv 1$ as in \eqref{eq:def_rstar}. Then
\[
	r^*(\omega, \alpha) = \frac{1}{2\alpha} r(2\alpha \omega).
\]
\end{proposition}

\section{Concluding remarks}

We have seen that Theorem \ref{th3.2} of D.~Wei and Theorem \ref{thm:Riccati} of the first two authors of the current work are optimal when applied to the differentiation operator on an interval in \eqref{eq:3.1}. The authors are not aware, however, of an example showing that the upper bound in Theorem \ref{th3.2} is optimal for $t > \frac{\pi}{2r(0)}$.

There are also simple examples (such as semigroups generated by Jordan blocks) such that the upper bounds of Theorem \ref{thm:Riccati} are not optimal. In certain situations, using the semigroup property for norms (Section \ref{s:semigroupization}) and iteration (Sections \ref{s:semigroupization} and \ref{s:paffine}) can offer some modest improvements, but a significant gap remains. The authors would naturally be interested in results which narrow this gap, either by improving Theorem \ref{thm:Riccati} or finding examples with large semigroup norms (and appropriate bounds on the resolvents of generators).

\end{document}